\documentclass[11pt]{amsart}
\usepackage{graphicx, amssymb}
\numberwithin{equation}{section}

\newtheorem{Thm}{Theorem}[section]

\newtheorem{Prop}[Thm]{Proposition}

\newtheorem{Lem}[Thm]{Lemma}
\newtheorem{Def}[Thm]{Definition}

\def\bD {\mathbf{D}}
\def\bF {\mathbf{F}}
\def\bG {\mathbf{G}}
\def\bN {\mathbf{N}}
\def\bR {\mathbf{R}}
\def\bS {\mathbf{S}}

\def\bZ {\mathbf{Z}}

\def\cB {\mathcal{B}}

\def\cD {\mathcal{D}}
\def\cE {\mathcal{E}}

\def\cK {\mathcal{K}}
\def\cL {\mathcal{L}}
\def\cM {\mathcal{M}}

\def\cQ {\mathcal{Q}}

\def\wL {{w-L}}

\def\a {{\alpha}}
\def\b {{\beta}}
\def\g {{\gamma}}
\def\de {{\delta}}
\def\eps {{\epsilon}}
\def\th {{\theta}}
\def\ka {{\kappa}}
\def\l {{\lambda}}

\def\om {{\omega}}
\def\Om {{\Omega}}

\def\rstr {{\big |}}
\def\indc {{\bf 1}}

\def\la {\langle}
\def\ra {\rangle}
\def \La {\bigg\langle}
\def \Ra {\bigg\rangle}
\def \lA {\big\langle \! \! \big\langle}
\def \rA {\big\rangle \! \! \big\rangle}
\def \LA {\bigg\langle \! \! \! \! \! \; \bigg\langle}
\def \RA {\bigg\rangle \! \! \! \! \! \; \bigg\rangle}

\def\bFc {{\bF^{\hbox{conv}}_\eps}}
\def\bFcn{{\bF^{\hbox{conv}}_{\eps_n}}}
\def\bFd {{\bF^{\hbox{diff}}_\eps}}
\def\bFdn{{\bF^{\hbox{diff}}_{\eps_n}}}

\def\d {{\partial}}
\def\grad {{\nabla}}
\def\Dlt {{\Delta}}

\newcommand{\Div}{\operatorname{div}}

\newcommand{\Span}{\operatorname{span}}
\newcommand{\Supp}{\operatorname{supp}}

\newcommand{\Tr}{\operatorname{trace}}

\newcommand{\Ker}{\operatorname{Ker}}

\newcommand{\ba}{\begin{aligned}}
\newcommand{\ea}{\end{aligned}}

\newcommand{\be}{\begin{equation}}
\newcommand{\ee}{\end{equation}}

\newcommand{\lb}{\label}

\begin{document}

\title[Navier-Stokes Limit of the Boltzmann equation]
      {The Incompressible Navier-Stokes Limit\\ 
      of the Boltzmann Equation for Hard Cutoff Potentials}

\author[F. Golse]{Fran\c cois Golse}
\address[F. G.]%
{Ecole polytechnique\\
Centre de math\'ematiques L. Schwartz\\
F91128 Palaiseau cedex} 

\email{golse@math.polytechnique.fr}

\address[F. G.]%
{Universit\'e Paris Diderot - Paris 7\\
Laboratoire J.-L. Lions\\
4 place Jussieu\\
Bo\^\i te courrier 187\\
F75252 Paris cedex 05} 

\author[L. Saint-Raymond]{Laure Saint-Raymond}
\address[L. S.-R.]%
{Ecole Normale Sup\'erieure\\
D\'epartement de Math\'ematiques et Applications\\
45 rue d'Ulm\\
F75230 Paris cedex 05} 
\email{saintray@dma.ens.fr}

\begin{abstract}
The present paper proves that all limit points of sequences of renormalized solutions 
of the Boltzmann equation in the limit of small, asymptotically equivalent Mach and 
Knudsen numbers are governed by Leray solutions of the Navier-Stokes equations. 
This convergence result holds for hard cutoff potentials in the sense of H. Grad, and 
therefore completes earlier results by the same authors [Invent. Math. 155, 81-161 
(2004)] for Maxwell molecules.
\end{abstract}



\maketitle


\section{Introduction}


The subject matter of this article is the derivation of the Navier-Stokes equations for 
incompressible fluids from the Boltzmann equation, which is the governing equation 
in the kinetic theory of rarefied, monatomic gases. 

In the kinetic theory of gases founded by Maxwell and Boltzmann,  the state of a 
monatomic gas is described by the molecular number density in the single-body 
phase space, $f\equiv f(t,x,v)\ge 0$ that is the density with respect to the Lebesgue 
measure $dxdv$ of molecules with velocity $v\in\bR^3$ and position $x\in\bR^3$ at 
time $t\ge 0$. Henceforth, we restrict our attention to the case where the gas fills the
Euclidian space $\bR^3$. For a perfect gas, the number density $f$ satisfies the 
Boltzmann equation
\be\lb{BoltzEq}
\d_tf+v\cdot\grad_xf=\cB(f,f)\,,\quad x,v\in\bR^3\,,
\ee
where $\cB(f,f)$ is the Boltzmann collision integral. 

The Boltzmann collision integral acts only on the $v$ variable in the number density
$f$. In other words, $\cB$ is a bilinear operator defined on functions of the single
variable $v$, and it is understood that the notation
\be
\label{NotCollInt}
\cB(f,f)(t,x,v)\hbox{ designates }\cB(f(t,x,\cdot),f(t,x,\cdot))(v)\,,
\ee
For each continuous $f\equiv f(v)$ rapidly decaying at infinity, the collision integral
is given by
\be
\lb{CollInt}
\cB(f,f)(v)=\iint_{\bR^3\times\bS^2}(f(v')f(v'_1)-f(v)f(v_1))b(v-v_1,\om)dv_1d\om
\ee
where 
\be
\lb{Frml-v'v}
\ba
v'&\equiv v'(v,v_1,\om)\,=v-(v-v_1)\cdot\om\om\,,
\\
v'_1&\equiv v'_1(v,v_1,\om)=v_1+(v-v_1)\cdot\om\om\,.
\ea
\ee
The collision integral is then extended by continuity to wider classes of densities 
$f$, depending on the specifics of the function $b$.

The function $b\equiv b(v-v_1,\om)$, called the collision kernel, is measurable, 
a.e. positive, and satisfies the symmetry 
\be
\lb{b-sym}
b(v-v_1,\om)=b(v_1-v,\om)=b(v'-v'_1,\om)\hbox{ a.e. in }(v,v_1,\om)\,.
\ee
Throughout the present paper, we assume that $b$ satisfies
\be
\lb{Grad-ctff}
\ba
0<b(z,\om)\le C_b(1+|z|)^\b|\cos(\widehat{z,\om})|\hbox{ a.e. on }\bR^3\times\bS^2\,,
\\
\int_{\bS^2}b(z,\om)d\om\ge\frac1{C_b}\frac{|z|}{1+|z|}\hbox{ a.e. on }\bR^3\,.
\ea
\ee
for some $C_b>0$ and $\b\in[0,1]$. The bounds (\ref{Grad-ctff}) are verified by 
all collision kernels coming from a repulsive, binary intermolecular potential of 
the form $U(r)=U_0/r^s$ with Grad's angular cutoff (see \cite{GradRGD62}) and 
$s\ge 4$. Such power-law potentials are said to be ``hard" if $s\ge 4$ and ``soft" 
otherwise: in other words, we shall be dealing with hard cutoff potentials. The 
case of a hard-sphere interaction (binary elastic collisions between spherical
particles) corresponds with
\be
\lb{HrdSph-b}
b(z,\om)=|z\cdot\om|\,;
\ee
it is a limiting case of hard potentials that obviously satisfies (\ref{Grad-ctff}), even 
without Grad's cutoff. At the time of this writing, the Boltzmann equation has been 
derived from molecular dynamics --- i.e. Newton's equations of classical  mechanics 
applied to a large number of spherical particles --- in the case of hard sphere 
collisions, by O.E. Lanford \cite{Lanford}, see also \cite{EspoPulviHdbook} for the 
case of compactly supported potentials. Thus the collision kernel $b$ given by 
(\ref{HrdSph-b}) plays an important role in the mathematical theory of the Boltzmann 
equation.

The only nonnegative, measurable number densities $f$ such that $\cB(f,f)=0$ 
are Maxwellian densities, i.e. densities of the form 
\be
\lb{Maxw}
f(v)=\frac{R}{(2\pi\Theta)^{3/2}}e^{Ð\frac{|v-U|^2}{2\Theta}}=:\cM_{R,U,\Theta}(v)
\ee
for some $R\ge 0$, $\Theta>0$ and $U\in\bR^3$. Maxwellian densities whose
parameters $R,U,\Theta$ are constants are called ``uniform Maxwellians", 
whereas Maxwellian densities whose parameters $R,U,\Theta$ are functions 
of $t$ and $x$ are referred to as ``local Maxwellians". Uniform Maxwellians 
are solutions of (\ref{BoltzEq}); however, local Maxwellians are not  solutions
of (\ref{BoltzEq}) in general.

The incompressible Navier-Stokes limit of the Boltzmann equation can be stated
as follows. 

\newpage
\noindent
{\sc Navier-Stokes Limit of the Boltzmann Equation}

Let $u^{in}\equiv u^{in}(x)\in\bR^3$ be a divergence-free vector field
on $\bR^3$. For each $\eps>0$, consider the initial number density
\be
\lb{PertMaxw0}
f^{in}_\eps(x,v)=\cM_{1,\eps u^{in}(\eps x),1}(v)\,.
\ee
Notice that the number density $f^{in}_\eps$ is a slowly varying perturbation of 
order $\eps$ of the uniform Maxwellian $\cM_{1,0,1}$. Let $f_\eps$ solve the
Boltzmann equation (\ref{BoltzEq}) with initial data (\ref{PertMaxw0}), and define
\be
\lb{u-Feps}
u_\eps(t,x):=
	\frac1\eps\int_{\bR^3}vf_\eps\left(\frac{t}{\eps^2},\frac{x}{\eps},v\right)dv\,.
\ee
Then, in the limit as $\eps\to 0^+$ (and possibly after extracting a converging
subsequence), the velocity field $u_\eps$ satisfies
$$
u_\eps\to u\hbox{ in }\cD'(\bR_+\times\bR^3)
$$
where $u$ is a solution of the incompressible Navier-Stokes equations 
\be
\lb{Incompr-NS}
\ba
\d_t u+\Div_x(u\otimes u)+\grad_xp&=\nu\Dlt_xu\,,\quad x\in\bR^3\,,\,\,t>0\,,
\\
\Div_xu&=0\,,
\ea
\ee
with initial data
\be
\lb{NS-InData}
u\rstr_{t=0}=u^{in}\,.
\ee
The viscosity $\nu$ is defined in terms of the collision kernel $b$, by some
implicit formula, that will be given below.

\smallskip
(More general initial data than (\ref{PertMaxw0}) can actually be handled
with our method: see below for a precise statement of the Navier-Stokes
limit theorem.)

\smallskip
Hydrodynamic limits of the Boltzmann equation leading to incompressible
fluid equations have been extensively studied by many authors. See in
particular \cite{BGL1} for formal computations, and \cite{BGL0,BGL2} for 
a general program of deriving global solutions of incompressible fluid
models from global solutions of the Boltzmann equation. The derivation
of global weak (Leray) solutions of the Navier-Stokes equations from 
global weak (renormalized \`a la DiPerna-Lions) solutions of the Boltzmann
equation is presented in \cite{BGL2}, under additional assumptions on 
the Boltzmann solutions which remained unverified. In a series of later 
publications \cite{LiMa,LiMa2,BGL4,GoLe} some of these assumptions 
have been removed, except one that involved controlling the build-up of 
particles with large kinetic energy, and possible concentrations in the  
$x$-variable. This last assumption was
removed by the second author in the case of the model BGK equation 
\cite{SR1,SR2}, by a kind of dispersion argument based on the fact that
relaxation to local equilibrium improves the regularity in $v$ of number
density fluctuations. Finally, a complete proof of the Navier-Stokes limit
of the Boltzmann equation was proposed in \cite{GSRInvMath}. In this paper,
the regularization in $v$ was obtained by a rather different argument --- 
specifically, by the smoothing properties of the gain part of Boltzmann's 
collision integral --- since not much is known about relaxation to local
equilibrium for weak solutions of the Boltzmann equation. 

While the results above holds for global solutions of the Boltzmann
equation without restriction on the size (or symmetries) of its initial 
data, earlier results had been obtained in the regime of smooth
solutions \cite{dMEL,BaUk}. Since the regularity of Leray solutions
of the Navier-Stokes equations in 3 space dimensions is not known
at the time of this writing, such results are limited to either local
(in time) solutions, or to solutions with initial data that are small in
some appropriate norm.

The present paper extends the result of \cite{GSRInvMath} to the case
of hard cutoff potentials in the sense of Grad --- i.e. assuming that
the collision kernel satisfies (\ref{Grad-ctff}). Indeed, \cite{GSRInvMath}
only treated the case of Maxwell molecules, for which the collision
kernel is of the form
$$
b(z,\om)=|\cos(z,\om)|b^*(|\cos(z,\om)|)
		\hbox{ with } \frac1{C_*}\le b_*\le C_*\,.
$$
The method used in the present paper also significantly simplifies 
the original proof in \cite{GSRInvMath} in the case of Maxwell molecules.

Independently, C.D. Levermore and N. Masmoudi have extended
the analysis of \cite{GSRInvMath} to a wider class of collision kernels
that includes soft potentials with a weak angular cutoff in the sense
of DiPerna-Lions: see \cite{LevMas}. Their proof is written in the
case where the spatial domain is the 3-torus $\bR^3/\bZ^3$. 

In the present paper, we handle the case of the Euclidian space 
$\bR^3$, which involves additional technical difficulties concerning
truncations at infinity and the Leray projection on divergence-free
vector fields --- see Appendix C below.


\section{Formulation of the problem and main results}


\subsection{Global solutions of the Boltzmann equation}

The only global existence theory for the Boltzmann equation
without extra smallness assumption on the size of the initial 
data known to this date is the R. DiPerna-P.-L. Lions theory
of renormalized solutions \cite{DPL,Li}. We shall present 
their theory in the setting best adapted to the hydrodynamic
limit considered in the present paper.

All incompressible hydrodynamic limits of the Boltzmann
equation involve some background, uniform Maxwellian
equilibrium state --- whose role from a physical viewpoint
is to set the scale of the speed of sound. Without loss of
generality, we assume this uniform equilibrium state to be
the centered, reduced Gaussian density
\be
\lb{RefMaxw}
M(v):=\cM_{1,0,1}(v)=\frac1{(2\pi)^{3/2}}e^{-|v|^2/2}\,.
\ee

Our statement of the Navier-Stokes limit of the Boltzmann 
equation given above suggests that one has to handle the
scaled number density
\be
\lb{ScldDens}
F_\eps(t,x,v)=f_\eps\left(\frac{t}{\eps^2},\frac{x}{\eps},v\right)
\ee
where $f_\eps$ is a solution of the Boltzmann equation
(\ref{BoltzEq}). This scaled number density is a solution of
the scaled Boltzmann equation
\be
\lb{ScldBoltzEq}
\eps^2\d_tF_\eps+\eps v\cdot\grad_xF_\eps=\cB(F_\eps,F_\eps)\,,
\quad x,v\in\bR^3\,,\,\,t>0\,.
\ee
Throughout the present section, $\eps$ is any fixed, positive
number.

\begin{Def}
A renormalized solution of the scaled Boltzmann equation 
(\ref{ScldBoltzEq})  relatively to the global equilibrium $M$ 
is a function
$$
F\in C(\bR^+, L^1_{loc}(\bR^3\times \bR^3))
$$
such that
$$
\Gamma'\left(\frac{F}{M}\right)\cB(F,F)
	\in L^1_{loc}(\bR_+\times\bR^3\times\bR^3)
$$
and which satisfies 
\be
\lb{boltz-renormalized}
M\left(\eps^2\d_t+\eps v\cdot\grad_x\right)
	\Gamma\left(\frac{F}{M}\right)
		=\Gamma'\left(\frac{F}{M}\right) \cB(F,F)
\ee
for each normalizing nonlinearity
$$
\Gamma\in C^1(\bR^+)\hbox{ such that }
	|\Gamma'(z)|\leq\frac{C}{\sqrt{1+z}}\,,\quad z\ge 0\,.
$$
\end{Def}
 
The DiPerna-Lions theory is based on the only a priori
estimates that have natural physical interpretation. In
particular, the distance between any number density 
$F\equiv F(x,v)$ and the uniform equilibrium $M$ is
measured in terms of the relative entropy
\be
\lb{RelEnt}
H(F|M):=\iint_{\bR^3\times\bR^3}
	\left(F\ln\left(\frac{F}{M}\right)-F+M\right)dxdv\,.
\ee
Introducing 
\be
\lb{Def-h}
h(z)=(1+z)\ln(1+z)-z\ge 0\,,\quad z>-1\,,
\ee
we see that
$$
H(F|M)=\iint_{\bR^3\times\bR^3}
	h\left(\frac{F}{M}-1\right)Mdvdx\ge 0
$$
with equality if and only if $F=M$ a.e. in $x,v$.

While the relative entropy measures the distance of
a number density $F$ to the particular equilibrium
$M$, the local entropy production rate ``measures the
distance" of $F$ to the set of all Maxwellian densities.
Its expression is as follows:
\be
\lb{dissipation-def}
\cE(F)=\tfrac14\iiint_{\bR^3\times \bR^3\times\bS^2} 
	(F'F'_1-FF_1)\ln\left(\frac{F'F'_1}{FF_1}\right)
		b(v-v_1,\om)dvdv_1d\om\,.
\end{equation}

The DiPerna-Lions existence theorem is the following
statement  \cite{DPL,Li}.

\begin{Thm} \lb{Boltz-existence}
Assume that the collision kernel $b$ satisfies Grad's cutoff 
assumption (\ref{Grad-ctff}) for some $\beta \in [0,1]$. Let 
$F^{in}\equiv F^{in}(x,v)$ be any measurable, a.e. 
nonnegative function on $\bR^3\times\bR^3$ such that
\be
H(F^{in}|M)<+\infty\,.
\ee
Then, for each $\eps>0$, there exists a renormalized solution 
$$
F_\eps\in C(\bR^+, L_{loc}^1(\bR \times \bR^3))
$$ 
relatively to $M$ of the scaled Boltzmann equation 
(\ref{ScldBoltzEq}) such that
$$
F_\eps\rstr_{t=0}=F^{in}\,.
$$
Moreover, $F_\eps$ satisfies 

(a) the continuity equation
\be
\lb{mass-conservation}
\eps\d_t\int_{\bR^3}F_\eps dv+\Div_x\int_{\bR^3}vF_\eps dv=0\,,
\ee
and

(b) the entropy inequality
\be
\lb{entropy-ineq}
H(F_\eps|M)(t)+\frac1{\eps^2}\int_0^t \int_{\bR^3}\cE(F_\eps)(s,x)dsdx
	\le H(F^{in}|M)\,,\quad t>0\,.
\end{equation}
\end{Thm}

\smallskip
Besides the continuity equation (\ref{mass-conservation}),
classical solutions of the scaled Boltzmann equation 
(\ref{ScldBoltzEq}) with fast enough decay as $|v|\to\infty$ 
would satisfy the local conservation of momentum
\be
\lb{mom-conservation}
\eps\d_t\int_{\bR^3}vF_\eps dv+\Div_x\int_{\bR^3}v\otimes vF_\eps dv=0\,,
\ee
as well as the local conservation of energy
\be
\lb{ener-conservation}
\eps\d_t\int_{\bR^3}\tfrac12|v|^2F_\eps dv
	+\Div_x\int_{\bR^3}v\tfrac12|v|^2F_\eps dv=0\,.
\ee
Renormalized solutions of the Boltzmann equation
(\ref{ScldBoltzEq}) are not known to satisfy any of these
conservation laws except that of mass --- i.e. the continuity 
equation (\ref{mass-conservation}). Since these local
conservation laws are the fundamental objects in every
fluid theory, we expect to recover them somehow in the
hydrodynamic limit $\eps\to 0^+$.

\subsection{The convergence theorem}

It will be more convenient to replace the number density $F_\eps$ 
by its ratio to the uniform Maxwellian equlibrium $M$; also we shall 
be dealing mostly with perturbations of order $\eps$ of the uniform 
Maxwellian state $M$. Thus we define
\be
\lb{Def-Gg}
G_\eps=\frac{F_\eps}M\,,\quad g_\eps=\frac{G_\eps-1}{\eps}\,.
\ee
Likewise, the Lebesgue measure $dv$ will be replaced with the unit
measure $Mdv$, and we shall systematically use the notation
\be
\lb{Not<>}
\la\phi\ra=\int_{\bR^3}\phi(v)M(v)dv\,,\hbox{ for each }\phi\in L^1(Mdv)\,.
\ee
For the same reason, quantities like the local entropy production
rate involve the measure 
\be
\lb{Def-mu}
d\mu(v,v_1,\om)=b(v-v_1,\om)M_1dv_1Mdvd\om\,,\quad
	\iiint_{\bR^3\times\bR^3\times\bS^2}d\mu(v,v_1,\om)=1\,,
\ee
whose normalization can be assumed without loss of generality, by
some appropriate choice of physical units for the collision kernel $b$.
We shall also use the notation
\be
\lb{Not<<>>}
\lA\psi\rA=\iiint_{\bR^3\times\bR^3\times\bS^2}\psi(v,v_1,\om)d\mu(v,v_1,\om)
\hbox{ for }\psi\in L^1(\bR^3\times\bR^3\times\bS^2,d\mu)\,.
\ee

From now on, we consider solutions of the scaled Boltzmann equation
(\ref{ScldBoltzEq}) that are perturbations of order $\eps$  about the 
uniform Maxwellian $M$. This is conveniently expressed in terms of the
relative entropy.

\begin{Prop}[Uniform a priori estimates]\label{fluct-control} 
Let $F_\eps^{in}\equiv F_\eps^{in}(x,v)$ be a family of measurable, a.e. 
nonnegative functions such that
\be
\lb{init-fluctuation}
\sup_{\eps>0}\frac1{\eps^2}H(F_\eps^{in}|M)=C^{in}<+\infty\,.
\ee
Consider  a family  $(F_\eps)$ of renormalized solutions  of the scaled 
Boltzmann equation (\ref{ScldBoltzEq}) with initial data
\be
\lb{init-data}
F_\eps\rstr_{t=0}=F_\eps^{in}\,.
\ee
Then

(a) the family of relative number density fluctuations $g_\eps$ satisfies
\begin{equation}
\label{entropy-bound}
\frac1{\eps^2}\int_{\bR^3}\la h(\eps g_\eps(t,x,\cdot))\ra dx\le C^{in}
\end{equation}
where $h$ is the function defined in (\ref{Def-h});

(b) the family $\frac1\eps(\sqrt{G_\eps}-1)$ is bounded in 
$L^\infty(\bR_+;L^2(Mdvdx))$: 
\begin{equation}
\label{Entr-estm2}
\int_{\bR^3}\La\left(\frac{\sqrt{G_\eps}-1}{\eps}\right)^2\Ra dx\le C^{in}\,;
\end{equation}

(c) hence the family $g_\eps$ is relatively compact in
$L^1_{loc}(dtdx;L^1(Mdv))$;

(d) the family of relative number densities $G_\eps$ satisfies the
entropy production --- or dissipation estimate
\begin{equation}
\label{Entr-prd1}
\int_0^\infty\int_{\bR^3}\LA\left(
\frac{\sqrt{G'_\eps G'_{\eps 1}}-\sqrt{G_\eps G_{\eps 1}}}{\eps^2}
\right)^2\RA dxdt\le C^{in}\,.
\end{equation}
\end{Prop}

\begin{proof}
The entropy inequality implies that
$$
H(F_\eps|M)(t)=\int_{\bR^3}\la h(G_\eps-1)\ra (t,x)dx
	\le H(F_\eps^{in}|M)\le C^{in}\eps^2\,,
$$
which is the estimate (a).

The estimate (b) follows from (a) and the elementary identity
$$
\begin{aligned}
h(z-1)-(\sqrt{z}-1)^2&=z\ln z-(\sqrt z-1)(\sqrt z+1)-(\sqrt{z}-1)^2
\\
&=2z\ln\sqrt z-2(\sqrt z-1)\sqrt z
\\
&=2\sqrt z\left(\sqrt z\ln\sqrt z-\sqrt z+1\right)\ge 0
\end{aligned}
$$

From the identity
\begin{equation}
\label{identity1}
g_\eps=2\frac{\sqrt{G_\eps}-1}{\eps} 
	+\eps\left(\frac{\sqrt{G_\eps}-1}{\eps}\right)^2
\end{equation}
and the bound (b), we deduce the weak compactness statement (c).

Finally, the entropy inequality implies that
$$
\int_0^\infty\int_{\bR^3}\cE(F_\eps)(s,x)dxds\le C^{in}\eps^4\,.
$$
Observing that
$$
\ba
\cE(F_\eps)=\tfrac14\iiint_{\bR^3\times\bR^3\times\bS^2}
(F'_\eps F'_{\eps 1}-F_\eps F_{\eps 1})
\ln\left(\frac{F'_\eps F'_{\eps 1}}{F_\eps F_{\eps 1}}\right)
b(v-v_1,\om)dvdv_1d\om
\\
=\tfrac14\LA
(G'_\eps G'_{\eps 1}-G_\eps G_{\eps 1})
\ln\left(\frac{G'_\eps G'_{\eps 1}}{G_\eps G_{\eps 1}}\right)\RA
\ea
$$
and using the elementary inequality
$$
\tfrac14(X-Y)\ln\frac{X}{Y}\ge(\sqrt{X}-\sqrt{Y})^2\,,\quad X,Y>0
$$
leads to the dissipation estimate (d).
\end{proof}

\bigskip
Our main result in the present paper is a description of all limit points 
of the family of number density fluctuations $g_\eps$.

\begin{Thm}\label{BNSW-TH}
Let $F_{\eps}^{in}$ be a family of measurable, a.e. nonnegative functions 
defined on $\bR^3\times\bR^3$ satisfying the scaling condition 
(\ref{init-fluctuation}). Let  $F_\eps$ be a family of renormalized solutions 
relative to $M$ of the scaled Boltzmann equation (\ref{ScldBoltzEq}) with 
initial data (\ref{init-data}), for a hard cutoff collision kernel $b$ that satisfies 
(\ref{Grad-ctff}) with $\beta \in [0,1]$. Define the relative number density 
$G_\eps$ and the number density fluctuation $g_\eps$ by the formulas
(\ref{Def-Gg}).

Then, any limit point $g$ in $L^1_{loc}(dtdx;L^1(Mdv))$ of the family of 
number density fluctuations $g_\eps$  is an infinitesimal Maxwellian of 
the form
$$
g(t,x,v)=u(t,x)\cdot v+\th(t,x)\tfrac12(|v|^2-5)\,, 
$$
where the vector field $u$ and the function $\th$ are solutions of the
Navier-Stokes-Fourier system
\be
\lb{NSF}
\ba
\d_tu+\Div_x(u\otimes u)+\grad_xp&=\nu\Dlt_xu\,,\quad\Div_xu=0
\\
\d_t\th+\Div_x(u\th)&=\ka\Dlt_x\th
\ea
\ee
with initial data
\be
\lb{NSF-condin}
\ba
u^{in}&=w-\lim_{\eps\to 0}P\left(\frac1\eps\int vF^{in}_\eps dv\right)
\\
\th^{in}&=w-\lim_{\eps\to 0}\frac1\eps\int(\tfrac15|v|^2-1)(F^{in}_\eps-M)dv\,,
\ea
\ee
where $P$ is the Leray orthogonal projection in $L^2(\bR^3)$ on the space 
of diver- gence-free vector fields and the weak limits above are taken along
converging subsequences. Finally, the weak solution $(u,\th)$ of (\ref{NSF})
so obtained satisfies the energy inequality
\be
\lb{EnergIneq}
\ba
\int_{\bR^3}(\tfrac12|u(t,x)|^2+\tfrac54|\th(t,x)|^2)dx
&+
\int_0^t\int_{\bR^3}(\nu|\grad_xu|^2+\tfrac52\ka|\grad_x\th|^2)dx
\\
&\le\varliminf_{\eps\to 0^+}\frac1{\eps^2} H(F_\eps^{in}|M)
\ea
\ee
The viscosity $\nu$ and thermal conductivity $\ka$ are defined implicitly in 
terms of the collision kernel $b$ by the formulas (\ref{nu-kappa0}) below.
\end{Thm}

\smallskip
There are several ways of stating the formulas giving $\nu$ and $\ka$.
Perhaps the quickest route to arrive at these formulas is as follows.

Consider the Dirichlet form associated to the Boltzmann collision integral
linearized at the uniform equilibrium $M$:
\be
\lb{DirForm}
\cD_M(\Phi):=\tfrac18\lA|\Phi'+\Phi'_1-\Phi-\Phi_1|^2\rA\,.
\ee
The notation $|\cdot|^2$ designates the Euclidian norm on $\bR^3$ 
when $\Phi$ is vector-valued, or the Frobenius norm on $M_3(\bR)$
(defined by $|A|=\Tr(A^*A)^{1/2}$) when $\Phi$ is matrix-valued. Let
$\cD^*$ be the Legendre dual of $\cD$, defined by the formula
$$
\cD^*(\Psi):=\sup_{\Phi}\left(\la\Psi\cdot\Phi\ra-\cD(\Phi)\right)
$$
where the notation $\Phi(v)\cdot\Psi(v)$ designates the Euclidian
inner product in $\bR^3$ whenever $\Phi,\Psi$ are vector valued,
or the Frobenius inner product in $M_3(\bR)$ whenever $\Phi,\Psi$
are matrix-valued (the Frobenius inner product being defined by
$A\cdot B=\Tr(A^*B)$.) 

With these notations, one has
\be
\lb{nu-kappa0}
\nu:=\tfrac15\cD^*(v\otimes v-\tfrac13|v|^2I)\,,\quad
\ka:=\tfrac4{15}\cD^*(\tfrac12v(|v|^2-5))\,.
\ee

\smallskip
The weak solutions of the Navier-Stokes-Fourier system obtained in 
Theorem \ref{BNSW-TH} satisfy the energy inequality (\ref{EnergIneq})
and thus are strikingly similar to Leray solutions of the Navier-Stokes
equations in 3 space dimensions --- of which they are a generalization.
The reader is invited to check that, whenever the initial data $F_\eps^{in}$
is chosen so that 
$$
\frac1{\eps^2}H(F_\eps^{in}|M)\to\tfrac12\int_{\bR^3}|u^{in}(x)|^2dx
\hbox{ as }\eps\to 0^+\,,
$$
then the vector field $u$ obtained in Theorem \ref{BNSW-TH} is indeed
a Leray solution of the Navier-Stokes equations. More information on
this kind of issues can be found in \cite{GSRInvMath}. See in particular the
statements of Corollary 1.8 and Theorem 1.9 in \cite{GSRInvMath}, which
hold verbatim in the case of hard cutoff potentials considered in the 
present paper, and which are deduced from  Theorem \ref{BNSW-TH}
as explained in \cite{GSRInvMath}.

\subsection{Mathematical tools and notations for the hydrodynamic limit}

An important feature of the Boltzmann collision integral is the following
symmetry relations (the collision symmetries). These collision symmetries 
are straightforward, but fundamental consequences of the identities
(\ref{b-sym}) verified by the collision kernel, and can be formulated in the 
following manner. Let $\Phi\equiv\Phi(v,v_1)$ be such that 
$\Phi\in L^1(\bR^3\times\bR^3\times\bS^2,d\mu)$. Then
\be
\lb{CollSym}
\ba
\iiint_{\bR^3\times\bR^3\times\bS^2}\Phi(v,v_1)d\mu(v,v_1,\om)
=
\iiint_{\bR^3\times\bR^3\times\bS^2}\Phi(v_1,v)d\mu(v,v_1,\om)
\\
=
\iiint_{\bR^3\times\bR^3\times\bS^2}
	\Phi(v'(v,v_1,\om),v'_1(v,v_1,\om))d\mu(v,v_1,\om)
\ea
\ee
where $v'$ and $v'_1$ are defined in terms of $v,v_1,\om$ by the
formulas (\ref{Frml-v'v}).

Since the Navier-Stokes limit of the Boltzmann equation is a statement
on number density fluctuations about the uniform Maxwellian $M$, it is
fairly natural to consider the linearization at $M$ of the collision integral.

First, the quadratic collision integral is polarized into a symmetric bilinear
operator, by the formula
$$
\cB(F,G):=\tfrac12\left(\cB(F+G,F+G)-\cB(F,F)-\cB(G,G)\right)\,.
$$

\smallskip
\underbar{The linearized collision integral} is defined as
\be
\lb{L-def}
\cL f=-2M^{-1}\cB(M,Mf)\,.
\ee

Assuming that the collision kernel $b$ comes from a hard cutoff potential
in the sense of Grad (\ref{Grad-ctff}), one can show (see \cite{GradRGD62} 
for instance) that $\cL$ is  a possibly 
unbounded, self-adjoint, nonnegative Fredholm operator on the Hilbert
space $L^2(\bR^3,Mdv)$ with domain 
$$
D(\cL)=L^2(\bR^3,a(|v|)^2Mdv)
$$
 and nullspace
\be
\lb{kerL}
\Ker\cL=\Span\{1,v_1,v_2,v_3,|v|^2\}\,,
\ee
and that $\cL$ can  be decomposed as
$$
\cL g (v)=a(|v|)g(v)-\cK g (v)
$$
where $\cK$ is a compact integral operator on $L^2(Mdv)$ and $a=a(|v|)$ is a 
scalar function called the collision frequency that satisfies, for some $C>1$,
$$
0<a_- \leq a(|v|)\leq a_+(1+|v|)^\beta \,.
$$
In particular, $\cL$ has a spectral gap, meaning that there exists $C>0$ such that
\be
\lb{Crcv}
\la f\cL f\ra\ge C\|f-\Pi f\|^2_{L^2( Madv)}\,;
\ee
for each $f\in D(\cL)$, where $\Pi$ is the orthogonal projection on $\Ker\cL$ 
in $L^2(\bR^3,Mdv)$, i.e.
\be
\lb{Dfnt-Pi}
\Pi f=\la f\ra+\la vf\ra\cdot v+\la(\tfrac13|v|^2-1)f\ra\tfrac12(|v|^2-3)\,.
\ee

\smallskip
\underbar{The bilinear collision integral} intertwined with the
multiplication by $M$ is defined by
\begin{equation}
\label{Q-def}
\cQ(f,g) =M^{-1}\cB(Mf,Mg)\,.
\end{equation}

Under the only assumption that the collision
kernel satisfies (\ref{b-sym}) together with the bound
\be
\lb{BoundAv-b}
\int_{\bS^2}b(z,\om)d\om\le a_+(1+|z|)^\b\,,
\ee
$\cQ$  maps continuously $L^2(\bR^3,M(1+|v|)^\beta dv)$ into 
$L^2(\bR^3,a^{-1}Mdv)$. Indeed, by using the Cauchy-Schwarz inequality 
and the collision symmetries (\ref{CollSym}) entailed by (\ref{b-sym})
\be
\lb{Cnt-Q}
\begin{aligned}
{}&\|\cQ(g,h)\|_{L^2(a^{-1}Mdv)}^2=\int_{\bR^3}a(|v|)^{-1}
\\
&\times\left(\tfrac12\iint_{\bR^3\times\bS^2}
	(g'h'_1+g'_1h'-gh_1-g_1h)b(v-v_1,\om)M_1dv_1d\om\right)^2Mdv
\\
&\le\tfrac12\int_{\bR^3}a(|v|)^{-1} \left(\iint_{\bR^3\times\bS^2}b(v-v_1,\om)M_1dv_1d\om\right)
\\
&\times\left(\iint_{\bR^3\times\bS^2}(g'h'_1+g'_1h'-gh_1-g_1h)^2
	b(v-v_1,\om)M_1dv_1d\om\right)Mdv
\\
&\le
\sup_{v\in\bR^3}a(|v|)^{-1} \iint_{\bR^3\times\bS^2}b(v-v_1,\om)M_1dv_1d\om
\\
&\times\iiint_{\bR^3\times\bR^3\times\bS^2} 
	((g'h'_1)^2+(g'_1h')^2 +(gh_1)^2 +(g_1h)^2)d\mu(v,v_1,\omega)
\\
&\le 2C\iint_{\bR^3\times\bR^3}((gh_1)^2+(g_1h)^2)
	\left(\int_{\bS^2}b(v-v_1,\omega)d\omega\right)MM_1dvdv_1
\\
&\le 4C^2\|g\|^2_{L^2((1+|v|)^{\b}Mdv)}\| h\|^2_{L^2((1+|v|)^{\b}Mdv)}\,.
\end{aligned}
\ee
Another important property of the bilinear operator $\cQ$ is the following 
relation:
\be
\lb{QKerL}
\cQ(f,f)=\tfrac12\cL(f^2)\hbox{ for each }f\in\Ker\cL\,,
\ee
which follows from differentiating twice both sides of the equality
$$
\cB(\cM_{R,U,\Theta},\cM_{R,U,\Theta})=0
$$
with respect to $R\ge 0$, $\Theta>0$ and $U\in\bR^3$ --- see for instance
\cite{BGL1}, flas (59-60) for a quick proof of this identity.

\smallskip
\underbar{Young's inequality}

\noindent
Since the family of number density fluctuations $g_\eps$ satisfies the
uniform bound (a) in Proposition \ref{fluct-control} and the measure
$Mdv$ has total mass $1$, the fluctuation $g_\eps$ can be integrated
against functions of $v$ with at most quadratic growth at infinity, by
an argument analogous to the H\"older inequality. This argument will
be used in various places in the proof, and we present it here for the
reader's convenience. To the function $h$ in (\ref{Def-h}), we associate
its Legendre dual $h^*$ defined by
$$
h^*(\zeta):=\sup_{z>-1}(\zeta z-h(z))=e^{\zeta}-\zeta-1\,.
$$
Thus, for each $\zeta>0$ and each $z>-1$, one has
\be
\lb{Young<}
\zeta|z|\le h(|z|)+h^*(\zeta)\le h(z)+h^*(\zeta)
\ee
since 
$$
h(|z|)\le h(z)\,,\quad z>-1\,.
$$
The inequality (\ref{Young<}) is referred to as the Young inequality
(by analogy with the classical Young inequality
$$
\zeta z\le\frac{z^p}{p}+\frac{\zeta^q}{q}\,,\quad z,\zeta>0
$$
which holds whenever $1<p,q<\infty$ satisfy $\frac1p+\frac1q=1$.) 

\smallskip
\underbar{ Notations regarding functional spaces}

\noindent
Finally, we shall systematically use the following notations. First,
Lebesgue spaces without mention of the domain of integration
always designate that Lebesgue space on the largest domain of
integration on which the measure is defined. For instance
$$
\ba
{}&L^p(Mdv)&&\hbox{ designates }L^p(\bR^3;Mdv)
\\
&L^p(Mdvdx)&&\hbox{ designates }L^p(\bR^3\times\bR^3;Mdvdx)
\\
&L^p(d\mu)&&\hbox{ designates }L^p(\bR^3\times\bR^3\times\bS^2;d\mu)
\ea
$$
When the measure is the Lebesgue measure, we shall simply 
denote
$$
L^p_x:=L^p(\bR^3;dx)\,,\quad L^p_{t,x}:=L^p(\bR_+\times\bR^3;dtdx)\,.
$$

Whenever $E$ is a normed space, the notations $O(\de)_E$ and $o(\de)_E$
designate a family of elements of $E$ whose norms are $O(\de)$ or $o(\de)$. 
(For instance $O(1)_E$ designates a bounded family in $E$, while $o(1)_E$
designates a sequence that converges to $0$ in $E$.)

Although $L^p_{loc}$ spaces are not normed spaces, we designate by the
notation $O(\de)_{L^p_{loc}(\Om)}$ a family $f_\eps\in L^p_{loc}(\Om)$ such
that, for each compact $K\subset\Om$,
$$
\|f_\eps\|_{L^p(K)}=O(\de)\,.
$$
The notation $o(\de)_{L^p_{loc}(\Om)}$ is defined similarly.

\subsection{Outline of the proof of Theorem \ref{BNSW-TH}}

In terms of the fluctuation $g_\eps$, the scaled Boltzmann equation 
(\ref{ScldBoltzEq}) with initial condition (\ref{init-data}) can be put in the
form 
\be
\lb{BoltzFluct}
\ba
\eps\d_tg_\eps+v\cdot\grad_xg_\eps&=-\frac1\eps\cL(g_\eps)+\cQ(g_\eps,g_\eps)\,,
\\
g_{\eps |t=0}&=g_\eps^{in}\,.
\ea
\ee

\noindent 
\underline{\it Step 1}~:
We first prove that any limit point $g$ of the family of fluctuations $g_\eps$ 
as $\eps\to 0^+$ satisfies
$$
g=\Pi g
$$
where $\Pi$ is the orthogonal projection on the nullspace of $\cL$ defined
in (\ref{Dfnt-Pi}). 

Hence, the limiting fluctuation $g$ is an infinitesimal Maxwellian, i.e. of
the form
\be
\lb{InfinitMaxw}
g(t,x,v)=\rho(t,x)+u(t,x)\cdot v+\th(t,x)\tfrac12(|v|^2-3)\,.
\ee
The limiting form of the continuity equation (\ref{mass-conservation}) is
equivalent to the incompressibility condition on $u$:
$$
\Div_xu=0\,.
$$

\noindent 
\underline{\it Step 2}~:
In order to obtain equations for the moments 
$$
\rho=\la g\ra\,,\quad u=\la vg \ra\,,\quad\hbox{ and }\th=\la (\tfrac13|v|^2-1) g\ra
$$
we pass to the limit in approximate local conservation laws deduced from
the Boltzmann equation in the following manner.

Besides the square-root renormalization, we  use a renormalization of the 
scaled Boltzmann equation (\ref{ScldBoltzEq}) based on a smooth truncation 
$\g$ such that
\be
\lb{Dfnt-gmm}
\g\in C^\infty(\bR^+,[0,1])\,,\,\,\g\rstr_{[0,\tfrac32]}\equiv 1\,,
	\,\,\g\rstr_{[2,+\infty)}\equiv 0\,.
\ee
Define
\be
\lb{Dfnt-hgmm}
\hat\g(z)=\frac{d}{dz}((z-1)\g(z))\,.
\ee
Notice that
\be
\lb{Prpt-hgmm}
\Supp(\hat\g)\subset[0,2]\,,\,\,\hat\g\rstr_{[0,\tfrac32]}\equiv 1\,,
    \hbox{ and }\|1-\hat\g\|_{L^\infty}\le 1+\|\g'\|_{L^\infty}\,.
\ee
We use below the notation $\g_\eps$ and $\hat\g_\eps$ to denote respectively 
$\g(G_\eps)$ and $\hat\g(G_\eps)$.

We also use a truncation of high velocities, defined as follows: given ${\bf k}>6$, we 
set 
\be
\lb{DefKeps}
K_\eps={\bf k}|\ln\eps|\,.
\ee
For each continuous scalar function, or vector- or tensor-field $\xi\equiv\xi(v)$, 
we denote by $\xi_{K_\eps}$ the following truncation of $\xi$:
\be
\lb{TruncVel}
\xi_{K_\eps}(v)=\xi(v)\indc_{|v|^2\le K_\eps}\,.
\ee

Renormalizing the scaled Boltzmann equation (\ref{ScldBoltzEq}) with the
nonlinearity $\Gamma(Z)=(Z-1)\g(Z)$, we arrive at the following form of  
(\ref{BoltzFluct}) 
$$
\d_t(g_\eps\g_\eps)+\frac1\eps v\cdot\grad_x(g_\eps\g_\eps)
    =\frac1{\eps^3}\hat\g_\eps\cQ(G_\eps,G_\eps)\,.
$$
Multiplying each side of the equation above by $\xi_{K_\eps}$, and averaging 
in the variable $v$ leads to
\be
\lb{Mmnt-xi}
\d_t\la\xi_{K_\eps}g_\eps\g_\eps\ra
+
\Div_x\frac1\eps\la v\xi_{K_\eps}g_\eps\g_\eps\ra
=\frac1{\eps^3}
\lA\xi_{K_\eps}\hat\g_\eps(G'_\eps G'_{\eps 1}-G_\eps G_{\eps 1})\rA\,.
\ee

Henceforth we use the following notations for the fluxes of momentum or
energy:
\be
\lb{Dfnt-F}
\bF_\eps(\zeta)=\frac1\eps\la\zeta_{K_\eps}g_\eps\g_\eps\ra
\ee
with
$$
\zeta(v)=A(v):=v^{\otimes 2}-\tfrac13|v|^2I\,,
	\hbox{ or }\zeta(v)=B(v):=\tfrac12v(|v|^2-5)\,.
$$
Likewise, we use the notation
\be
\lb{Dfnt-D}
\bD_\eps(\xi)=\frac1{\eps^3}
\lA\xi_{K_\eps}\hat\g_\eps(G'_\eps G'_{\eps 1}-G_\eps G_{\eps 1})\rA
\ee
for the conservation defect corresponding with the (truncated) quantity
$\xi\equiv\xi(v)$, where $\xi\in\Span\{1,v_1,v_2,v_3,|v|^2\}$.

The Navier-Stokes motion equation is obtained by passing to the limit 
as $\eps\to 0$ modulo gradient fields in the equation (\ref{Mmnt-xi}) for 
$\xi(v)=v_j$, $j=1,2,3$, recast as
\be
\lb{Mmnt-v}
\d_t\la v_{K_\eps}g_\eps\g_\eps\ra+\Div_x\bF_\eps(A)
+\grad_x\la\tfrac13|v|^2_{K_\eps}g_\eps\g_\eps\ra=\bD_\eps(v)\,,
\ee
while the temperature equation is obtained by passing to the limit in 
that same equation with $\xi(v)=\tfrac12(|v|^2-5)$, i.e. in
\be
\lb{Mmnt-v2}
\d_t\la\tfrac12(|v|^2-5)_{K_\eps}g_\eps\g_\eps\ra
+\Div_x\bF_\eps(B)=\bD_\eps(\tfrac12(|v|^2-5))\,.
\ee

\bigskip
For the mathematical study of that limiting process, the uniform a priori 
estimates obtained from the scaled entropy inequality are not sufficient. 
Our first task is therefore to improve these estimates using both

\noindent
a)  the properties of the collision operator (see Section 3), namely a 
suitable control on the relaxation based on the coercivity estimate (\ref{Crcv})
$$\la\phi\cL\phi\ra\ge C\|\phi-\Pi\phi\|^2_{L^2( Madv)}\,;
$$

\noindent
b) and the properties of the free transport operator (see Section 4), 
namely dispersion and velocity averaging.

\smallskip
With the estimates obtained in Sections 3-4, we first prove (in Section 5)
that the conservation defects vanish asymptotically
$$
\bD_\eps(\xi) \to 0 \hbox{ in } L^1_{loc}(dtdx)\,,
	\quad \xi\in\Span\{v_1,v_2,v_3,|v|^2\}
$$

Next we analyze the asymptotic behavior of the flux terms. This requires 
splitting these flux terms into a convection and a diffusion part (Section 6)
$$
\bF_\eps(\zeta)-2\La\zeta\left(\Pi\frac{\sqrt{G_\eps}-1}{\eps}\right)^2\Ra
    +\frac2{\eps^2}\La\hat\zeta\cQ(\sqrt{G_\eps},\sqrt{G_\eps})\Ra
    	\to 0\hbox{ in }L^1_{loc}(dtdx)
$$
where $\hat \zeta$ is the unique solution in $(\Ker \cL)^\perp$ of the
Fredholm integral equation
$$
\cL \hat \zeta =\zeta\,.
$$
For instance, the tensor field $A$ and the vector field $B$ defined by
\be
\lb{Def-AB}
A(v):=v\otimes v-\tfrac13|v|^2I\,,\quad
	B(v):=\tfrac12(|v|^2-5)v
\ee
satisfy 
\be
\lb{AB-orthKer}
A\bot\Ker\cL\,,\quad B\bot\Ker\cL
\ee
componentwise, so that there exists a unique tensor field $\hat A$
and a unique vector field $\hat B$ such that
\be
\lb{LABhat}
\cL\hat A=A\,,\quad\cL\hat B=B\,,\quad\hat A\hbox{ and }\hat B\bot\Ker\cL\,,
\ee

The diffusion terms are easily proved to converge towards the dissipation
terms in the Navier-Stokes-Fourier system
$$
\frac2{\eps^2}\La\hat\zeta\cQ(\sqrt{G_\eps},\sqrt{G_\eps})\Ra
\to 
\la\hat\zeta(v\cdot\grad_xg)\ra\hbox{ in }L^1_{loc}(dtdx)\,.
$$
The formulas (\ref{nu-kappa0}) for the viscosity $\nu$ and heat conduction
$\kappa$ are easily shown to be equivalent to
\be
\lb{nu-kappa}
\nu=\tfrac1{10}\la\hat A:A\ra\,,\quad\kappa=\tfrac2{15}\la\hat B\cdot B\ra\,.
\end{equation}
The (nonlinear) convection terms require a more careful treatment, involving
 in particular some spatial regularity argument and the filtering of acoustic 
 waves (see Section 7).
 

\section{Controls on the velocity dependence\\ 
of the number density fluctuations}\label{UINTV-SCT}


The goal of this section is to prove that the square number density fluctuation 
--- or more precisely the following variant thereof
$$
\left(\frac{\sqrt{G_\eps}-1}{\eps}\right)^2
$$
is uniformly integrable in $v$ with the weight $(1+|v|)^p$ for each $p<2$. 

In our previous work \cite{GSRInvMath}, we obtained this type of control for 
$p=0$ only, by a fairly technical argument (see Section 6 of \cite{GSRInvMath}). 
Basically, we used the entropy production bound to estimate some notion of 
distance between the number density and the gain part of a fictitious collision 
integral. The conclusion followed from earlier results by Grad and Caflisch on 
the $v$-regularity of the gain term in Boltzmann's collision integral linearized 
at some uniform Maxwellian state. 

Unfortunately, this method seems to provide only estimates without the weight 
$(1+|v|)^\b$ (with $\b$ as in (\ref{Grad-ctff})) that is crucial for treating hard 
potentials other than the case of Maxwell molecules. Obtaining the weighted 
estimates requires some new ideas presented in this section.  

The first such idea is to use the  spectral gap estimate (\ref{Crcv}) for the 
linearized collision integral. Instead of comparing the number density to (some 
variant of) the local Maxwellian equilibrium --- as in the case of the BGK model
equation, treated in \cite{SR1,SR2}, or in the case of the Boltzmann equation 
with Maxwell molecules as in \cite{GSRInvMath} --- we directly compare the 
number density fluctuation to the infinitesimal Maxwellian that is its projection 
on hydrodynamic modes. 

The lemma below provides the basic argument for arriving at such estimates. 

\begin{Lem}\label{I-PI-LM}
Under the assumptions of Theorem \ref{BNSW-TH}, one has
\begin{equation}
\label{Crcv-1}
\left\|\frac{\sqrt{G_\eps}-1}{\eps}-\Pi\frac{\sqrt{G_\eps}-1}{\eps}\right\|_{L^2(Mdv)}
\le 
O(\eps)_{L^2_{t,x}}+O(\eps)\left\|\frac{\sqrt{G_\eps}-1}{\eps}\right\|^2_{L^2(Mdv)}\,.
\end{equation}
\end{Lem}

\begin{proof}
In order to simplify the presentation we first define some fictitious collision 
integrals $\tilde \cL$ and $\tilde \cQ$ 
$$
\ba
\tilde\cL g=\iint_{\bR^3\times\bS^2}
	(g+g_1-g'-g'_1)M_1 \tilde b(v-v_1,\omega)dv_1d\omega\,,
\\
\tilde\cQ(g,h)=\tfrac12\iint_{\bR^3\times\bS^2}
	(g'h'_1+g'_1h' -gh_1-g_1h) M_1\tilde b(v-v_1,\omega)dv_1d\omega,
\ea
$$
obtained from $\cL$ and $\cQ$ by replacing the original collision kernel $b$ 
with  
$$
\tilde b(z,\omega)=\frac{b(z,\omega)}{1+
	\displaystyle\int_{\bS^2}b(z,\omega_1) d\omega_1}\,.
$$

Start from the elementary formula
\be
\lb{Q-Exp}
\tilde\cL\left(\frac{\sqrt{G_\eps}-1}{\eps^2}\right)=
\tilde\cQ\left(\frac{\sqrt{G_\eps}-1}{\eps},\frac{\sqrt{G_\eps}-1}{\eps}\right)
-
\frac1{\eps^2}\tilde\cQ\left(\sqrt{G_\eps},\sqrt{G_\eps}\right)\,.
\ee
Multiplying both sides of this equation by $(I-\Pi)(\sqrt{G_\eps}-1)$ and using the  
spectral gap estimate (\ref{Crcv}) leads to
\be
\lb{Crcv-1stp}
\ba
\left\|\frac{\sqrt{G_\eps}-1}{\eps}
	-\Pi\frac{\sqrt{G_\eps}-1}{\eps}\right\|&_{L^2( Mdv)}
\\
&\le
\eps\left\|\tilde\cQ\left(\frac{\sqrt{G_\eps}-1}{\eps},
    	\frac{\sqrt{G_\eps}-1}{\eps}\right)\right\|_{L^2(Mdv)}
\\
&+
\eps\left\|\frac1{\eps^2}\tilde\cQ(\sqrt{G_\eps},\sqrt{G_\eps})\right\|_{L^2(Mdv)}\,.
\ea
\ee

Denote 
$$
d\tilde \mu(v,v_1,\omega)=MM_1 \tilde b(v-v_1,\omega) d\omega dvdv_1\,.
$$
By definition of $\tilde b$ , one has
$$
\int _{\bS^2}\tilde b(v-v_1,\omega)d\omega\le1\,.
$$
Hence $\cQ$ is continuous on $L^2(Mdv)$: by (\ref{Cnt-Q})
$$
\|\tilde \cQ(g,h)\|_{L^2(Mdv)}\le 2\|g\|_{L^2(Mdv)}\| h\|_{L^2(Mdv)}\,.
$$
(Notice that $\tilde b$ verifies (\ref{b-sym}) as does $b$).

Plugging this estimate in (\ref{Crcv-1stp}) leads to
\be
\lb{Crcv-2stp}
\ba
{}&\left\|\frac{\sqrt{G_\eps}-1}{\eps}-\Pi\frac{\sqrt{G_\eps}-1}{\eps}\right\|_{L^2( Mdv)}
\\
&\quad\le C\eps\left\|\frac{\sqrt{G_\eps}-1}{\eps}\right\|_{L^2( Mdv)}^2
 +
\eps\left\|\frac1{\eps^2}\cQ(\sqrt{G_\eps},\sqrt{G_\eps})\right\|_{L^2(Mdv)}
\ea
\ee

Finally, applying the Cauchy-Schwarz inequality as in the proof of (\ref{Cnt-Q}), 
one finds that
$$
\ba
\left\|\frac1{\eps^2}\tilde\cQ(\sqrt{G_\eps},\sqrt{G_\eps})\right\|^2_{L^2(Mdv)}
\le\left(\sup_{v\in\bR^3}\iint_{\bR^3\times\bS^2}\tilde b(v-v_1,\om)M_1dv_1d\om\right)
\\
\times
\frac1{\eps^4}\!\iiint_{\bR^3\times\bR^3\times\bS^2}\!\!
\left(\sqrt{G'_\eps G'_{\eps 1}}-\sqrt{G_\eps G_{\eps 1}}\right)^2\!d\tilde\mu(v,v_1,\om)
\\
\le\frac1{\eps^4}\!\iiint_{\bR^3\times\bR^3\times\bS^2}\!\!
\left(\sqrt{G'_\eps G'_{\eps 1}}-\sqrt{G_\eps G_{\eps 1}}\right)^2\!d\mu(v,v_1,\om)\,,
\ea
$$
since $0\le\tilde b\le b$. By the entropy production estimate (d) in Proposition
\ref{fluct-control}, the inequality above implies that
$$
\left\|\frac1{\eps^2}\tilde \cQ(\sqrt{G_\eps},\sqrt{G_\eps})\right\|_{L^2(Mdv)}
	=O(1)_{L^2_{t,x}}\,.
$$
This estimate and (\ref{Crcv-2stp}) entail the inequality (\ref{Crcv-1}).
\end{proof}

\smallskip
Notice  that we could have used directly $\cL$ and $\cQ$ instead of their 
truncated analogues $\tilde\cL$ and $\tilde\cQ$, obtaining bounds in 
weighted $L^2$ spaces by some loop argument, unfortunately much
more technical than the proof above.

\smallskip
The main result in this section --- and one of the key new estimate in this paper 
is

\begin{Prop}\label{UINTV-PR}
Under the assumptions of Theorem \ref{BNSW-TH}, for each $T>0$, each 
compact $K\subset\bR^3$, and each $p<2$, the family
$$
(1+|v|)^p\left(\frac{\sqrt{G_\eps}-1}{\eps}\right)^2
$$
is uniformly integrable in $v$ on $[0,T]\times K\times\bR^3$ with respect 
to the measure $dtdxMdv$. (This means that, for each $\eta>0$, there 
exists $\a>0$ such that, for each measurable $\varphi\equiv\varphi(x,v)$ 
verifying
$$
\|\varphi\|_{L^\infty_{x,v}}\le 1\hbox{ and }\|\varphi\|_{L^\infty_x(L^1_v)}\le\a\,,
$$
one has
$$
\int_0^T\int_K \int_{\bR^3}\varphi (1+|v|)^p
	\left(\frac{\sqrt{G_\eps}-1}{\eps}\right)^2Mdvdxdt\le\eta\,.)
$$
\end{Prop}

\begin{proof}

Start from the decomposition
\begin{equation}
\label{new-decomposition}
\begin{aligned}
 J:&=(1+|v|)^p\left(\frac{\sqrt{G_\eps}-1}{\eps}\right)^2 
 =\left(\frac{\sqrt{G_\eps}-1}{\eps}\right) 
 	(1+|v|)^p \Pi \left(\frac{\sqrt{G_\eps}-1}{\eps}\right)
\\
&+ (1+|v|)^{\frac p2} \left(\frac{\sqrt{G_\eps}-1}{\eps}\right) 
(1+|v|)^{\frac p2} \left(\!\! \left(\!\frac{\sqrt{G_\eps}-1}{\eps}\!\right)
	-\Pi \left(\!\frac{\sqrt{G_\eps}-1}{\eps}\!\right)\!\!\right)
\end{aligned}
\end{equation}

\bigskip
We recall from the entropy bound (b) in Proposition \ref{fluct-control} that
$$
\left(\frac{\sqrt{G_\eps}-1}{\eps}\right) =O(1)_{L^\infty_t(L^2(dxMdv))}
$$
so that, by definition (\ref{Dfnt-Pi}) of the hydrodynamic projection $\Pi$
\begin{equation}
\label{ContPiL2-L2vp}
\Pi\left(\frac{\sqrt{G_\eps}-1}{\eps}\right) 
=O(1)_{L^\infty_t(L^2_x(L^q(Mdv)))}
\end{equation}
for all $q<+\infty$. Therefore the first term in the right-hand side of 
(\ref{new-decomposition}) satisfies
\begin{equation}
\label{new-est1}
I=\left|\frac{\sqrt{G_\eps}-1}{\eps}\right| 
(1+|v|)^p \left|\Pi\frac{\sqrt{G_\eps}-1}{\eps}\right|
=O(1)_{L^\infty_t(L^1_x(L^r(Mdv)))}
\end{equation}
for all $0\le p<+\infty$ and $1\le r<2$.

\bigskip
In order to estimate the second term in the right-hand side of 
(\ref{new-decomposition}), we first remark that, for each $\delta>0$, 
each $p<2$ and each $q<+\infty$, there exists some $C=C(p,q,\delta)$ 
such that
\begin{equation}
\label{new-young}
(1+|v|)^{p/2} \left(\frac{\sqrt{G_\eps}-1}{\eps}\right) 
=O(\delta)_{L^\infty_t (L^2(dxMdv))} 
	+O\left(\frac{C(p,q,\delta)}{\eps}\right)_{L^\infty_{t,x} (L^q(Mdv))}
\end{equation}
Indeed, by Young's inequality and Proposition \ref{fluct-control} (a),
$$
\begin{aligned}
(1+|v|)^{p} \left(\frac{\sqrt{G_\eps}-1}{\eps}\right) ^2 
&\leq \frac{\delta^2} {\eps^2} |G_\eps-1| \left(\frac{(1+|v|)^p}{\delta^2}\right) 
\\
& \leq\frac{\delta^2}{\eps^2} h(G_\eps-1) 
	+\frac{\delta^2}{\eps^2} h^*\left(\frac{(1+|v|)^p}{\delta^2}\right) 
\\
&=O(\delta^2)_{L^\infty_t (L^1(dxMdv))} 
	+\frac{\delta^2}{\eps^2}\exp\left(\frac{(1+|v|)^p}{\delta^2}\right)
\end{aligned}
$$

We next use (\ref{new-young}) with the two following observations: first,
the obvious continuity statement (\ref{ContPiL2-L2vp}). Also, because of 
(\ref{Crcv-1}) and the entropy bound (b) in Proposition \ref{fluct-control}, 
one has
\begin{equation}
\label{Crcv-1'}
\left\|\frac{\sqrt{G_\eps}-1}{\eps}
	-\Pi\frac{\sqrt{G_\eps}-1}{\eps}\right\|_{L^2(Mdv)}
		=O(\eps)_{L^1_{loc}(dtdx)}
\end{equation}

Hence
$$
\begin{aligned}
(1+|v|)^{\frac p2} \left|\frac{\sqrt{G_\eps}-1}{\eps}\right| 
(1+|v|)^{\frac p2} \left|\!\! \left(\!\frac{\sqrt{G_\eps}-1}{\eps}\!\right)
	-\Pi \left(\!\frac{\sqrt{G_\eps}-1}{\eps}\!\right)\!\!\right|
\\
\le\frac{\delta}{\eps} h(G_\eps-1)^{1/2}
(1+|v|)^{p/2}\left|\!\! \left(\!\frac{\sqrt{G_\eps}-1}{\eps}\!\right)
	-\Pi \left(\!\frac{\sqrt{G_\eps}-1}{\eps}\!\right)\!\!\right|
\\
+ \frac{\delta}{\eps}(1+|v|)^{p/2}\exp\left(\frac{(1+|v|)^p}{2\delta^2}\right) 
	\left|\!\! \left(\!\frac{\sqrt{G_\eps}-1}{\eps}\!\right)
	 -\Pi \left(\!\frac{\sqrt{G_\eps}-1}{\eps}\!\right)\!\!\right|
\\
=:II+III\,.
\end{aligned}
$$

Now
$$
\begin{aligned}
II\le\frac{\delta}{2\eps^2} h(G_\eps-1)
+
\delta(1+|v|)^{p}\left|\Pi\frac{\sqrt{G_\eps}-1}{\eps}\right|^2
+
\delta(1+|v|)^{p}\left|\frac{\sqrt{G_\eps}-1}{\eps}\right|^2
\\
=O(\delta)_{L^\infty_t(L^1(Mdvdx))}+O(\delta)_{L^\infty_t(L^1(Mdvdx))}+\delta J
\end{aligned}
$$

On the other hand
$$
\begin{aligned}
\|III\|_{L^1_{loc}(dtdx;L^r(Mdv))}
\le 
\delta\left\|(1+|v|)^{p/2}\exp\left(\frac{(1+|v|)^p}{2\delta^2}\right)\right\|_{L^q(Mdv)}
\\
\times\left\|\frac1\eps\left(\!\frac{\sqrt{G_\eps}-1}{\eps}\!\right)
	 -\Pi \left(\!\frac{\sqrt{G_\eps}-1}{\eps}\!\right)\right\|_{L^1_{loc}(dtdx;L^2(Mdv))}
\\
=O(\delta C(p,q,\delta))
\end{aligned}
$$
with $r=\frac{2q}{q+2}$.

Putting all these controls together shows that
\be
\lb{BoundJ}
\begin{aligned}
J\le I+II+III
=O(1)_{L^\infty_t(L^1_x(L^r(Mdv)))}
\\
+O(\delta)_{L^\infty_t(L^1(Mdvdx))}
	+O(\delta)_{L^\infty_t(L^1(Mdvdx))}+\delta J
\\
+O(\delta C(p,q,\delta))_{L^1_{loc}(dtdx;L^r(Mdv))}
\end{aligned}
\ee
i.e.
$$
\begin{aligned}
(1-\delta)(1+|v|)^p\left(\frac{\sqrt{G_\eps}-1}{\eps}\right)^2
\le
O(1)_{L^\infty_t(L^1_x(L^r(Mdv)))} 
\\
+O(\delta C(p,q,\delta))_{L^1_{loc}(dtdx;L^r(Mdv))}
+O(\delta)_{L^\infty_t(L^1(Mdvdx))}
\end{aligned}
$$
which entails the uniform integrability in $v$ stated in Proposition \ref{UINTV-PR}.
\end{proof}

\smallskip
\noindent
\textbf{Remark:} replacing the estimate for II above with
$$
\begin{aligned}
II\le\frac{8\delta}{2\eps^2} h(G_\eps-1)
+
\frac{\delta}{8}(1+|v|)^{p}\left|\Pi\frac{\sqrt{G_\eps}-1}{\eps}\right|^2
+
\frac{\delta}{8}(1+|v|)^{p}\left|\frac{\sqrt{G_\eps}-1}{\eps}\right|^2
\\
=O(\delta)_{L^\infty_t(L^1(Mdvdx))}+O(\delta)_{L^\infty_t(L^1(Mdvdx))}
+\frac{\delta}{8} J
\end{aligned}
$$
and choosing $\de=4$ in (\ref{BoundJ}) shows that
$$
(1+|v|)^2\left(\frac{\sqrt{G_\eps}-1}\eps\right)^2
	\hbox{ is bounded in }L^1_{loc}(dtdx;L^1(Mdv))\,.
$$
In \cite{BGL2}, the Navier-Stokes limit of the Boltzmann equation is established 
assuming the uniform integrability in $[0,T]\times K\times\bR^3$ for the measure
$dtdxMdv$ of a quantity analogous to the one considered in this bound. As we shall 
see, the Navier-Stokes-Fourier limit of the Boltzmann equation is derived in the 
present paper by using only the weaker information in Proposition \ref{UINTV-PR}.


\section{Compactness results for the number density fluctuations}


The following result is the main technical step in the present paper.

\begin{Prop}\label{UINT-PR}
Under the assumptions in Theorem \ref{BNSW-TH}, for each $T>0$, each compact 
$K\subset\bR^3$ and each $p<2$, the family of functions
$$
\left(\frac{\sqrt{G_\eps}-1}{\eps}\right)^2(1+|v|)^p
$$
is uniformly integrable on $[0,T]\times K\times\bR^3$ for the measure $dtdxMdv$.
\end{Prop}

This Proposition is based on the uniform integrability in $v$ of that same quantity, 
established in Proposition \ref{UINTV-PR}, together with a bound on the streaming 
operator applied to (a variant of) the number density fluctuation (stated in Lemma 
\ref{ADVC-LM}). Except for some additional truncations, the basic principle of the 
proof is essentially the same as explained in  Lemma 3.6 of \cite{GSRInvMath} 
(which is recalled in Appendix B). In other words, while the result of Proposition 
\ref{UINTV-PR} provides some kind of regularity in $v$ only for the number density 
fluctuation, the bound on the free transport part of the Boltzmann equation gives 
the missing regularity (in the $x$-variable).

The technical difficulty comes from the fact that the square-root renormalization 
$\Gamma(Z)=\sqrt{Z}$ is not admissible for the Boltzmann equation due to the 
singularity at $Z=0$. We will therefore use an approximation of the square-root,
namely $z\mapsto\sqrt{z+\eps^\alpha}$ for some $\alpha\in ]1,2[$.

\begin{Lem}\label{ADVC-LM}
Under the assumptions in Theorem \ref{BNSW-TH}, for each $\a>0$, one has
$$
\begin{aligned}
(\eps\d_t+v\cdot\grad_x)\frac{\sqrt{\eps^\a+G_\eps}-1}{\eps}
&=
O(\eps^{2-\a/2})_{L^1(Mdvdxdt)}
\\
&+O(1)_{L^2((1+|v|)^{-\b}Mdvdxdt)}
\\
&+
O(\eps)_{L^1_{loc}(dtdx;L^2((1+|v|)^{-\b}Mdv))}\,.
\end{aligned}
$$
\end{Lem}

\begin{proof}[Proof of Lemma \ref{ADVC-LM}]
Start from the renormalized form  of the scaled Boltzmann equation (\ref{ScldBoltzEq}),
with normalizing function
$$
\Gamma_\eps(Z)=\frac{\sqrt{\eps^\a+Z}-1}{\eps}\,.
$$
This equation can be written as
\begin{equation}
\label{Nrml-Sqrt-Bltz}
(\eps\d_t+v\cdot\grad_x)\frac{\sqrt{\eps^\a+G_\eps}-1}{\eps}
=\frac1{\eps^2}\frac1{2\sqrt{\eps^\a+G_\eps}}\cQ(G_\eps,G_\eps)
=Q^1_\eps+Q^2_\eps\,,
\end{equation}
with
\begin{equation}
\label{Dfnt-Q1Q2}
\begin{aligned}
Q^1_\eps&\!=\!\frac1{\eps^2}\frac1{2\sqrt{\eps^\a+G_\eps}}
\\
&\times\iint\left(\sqrt{G'_\eps G'_{\eps 1}}-\sqrt{G_\eps G_{\eps 1}}\right)^2
b(v-v_1,\om)d\om M_1dv_1\,,
\\
Q^2_\eps&\!=\!\frac1{\eps^2}\frac1{\sqrt{\eps^\a+G_\eps}}
\\
&\times\iint\!\!\sqrt{G_\eps G_{\eps 1}}
\left(\!\!\sqrt{G'_\eps G'_{\eps 1}}-\sqrt{G_\eps G_{\eps 1}}\right)
b(v-v_1,\om)d\om M_1dv_1\,.
\end{aligned}
\end{equation}
The entropy production estimate (d) in Proposition \ref{fluct-control} and the obvious inequality
$$
\sqrt{\eps^\a+G_\eps}\ge\eps^{\a/2}
$$
imply that
\begin{equation}
\label{Estm-Q1}
\|Q^1_\eps\|_{L^1(Mdvdxdt)}\le \tfrac12C^{in}\eps^{2-\a/2}\,.
\end{equation}
On the other hand
$$
Q^2_\eps=\frac{\sqrt{G_\eps}}{\sqrt{\eps^\a+G_\eps}}\iint\sqrt{G_{\eps 1}}
\frac{\sqrt{G'_\eps G'_{\eps 1}}-\sqrt{G_\eps G_{\eps 1}}}{\eps^2}b(v-v_1,\om)d\om M_1dv_1\,.
$$
Write
$$
\sqrt{G_{\eps 1}}=1+\eps\frac{\sqrt{G_{\eps 1}}-1}{\eps}\,.
$$
Apply the Cauchy-Schwarz inequality as in the proof of (\ref{Cnt-Q}): then
$$
\ba
{}&\left\|\iint\sqrt{G_{\eps 1}}\frac{\sqrt{G'_\eps G'_{\eps 1}}-\sqrt{G_\eps G_{\eps 1}}}{\eps^2}
b(v-v_1,\om)M_1dv_1d\om\right\|_{L^2((1+|v|)^{-\b}Mdv)}
\\
&\le\sup_{v\in \bR^3} \!\!\left( \!(1+|v|)^{-\b}
\!\!\!\iint\!\!b(v-v_1,\omega)M_1dv_1d\om\!\!\right)^{1/2} 
\\
&\qquad\qquad\times\!
\LA\!\left(\!\frac{\sqrt{G'_\eps G'_{\eps 1}}-\sqrt{G_\eps G_{\eps 1}}}{\eps^2}\!\right)^2\!\RA^{1/2} 
\\
&+\eps
\sup_{v\in \bR^3} \left( (1+|v|)^{-\b} \iint M_1\left(\frac{\sqrt{G_{\eps 1}}-1}{\eps}\right)^2
b(v-v_1,\omega) dv_1 d\omega \right)^{1/2}
\\
&\qquad\qquad\times
\LA\left(\frac{\sqrt{G'_\eps G'_{\eps 1}}-\sqrt{G_\eps G_{\eps 1}}}{\eps^2}\right)^2\RA^{1/2} 
\ea
$$
Therefore
$$
\ba
{}&\left\|\iint\sqrt{G_{\eps 1}}\frac{\sqrt{G'_\eps G'_{\eps 1}}-\sqrt{G_\eps G_{\eps 1}}}{\eps^2}
b(v-v_1,\om)M_1dv_1\right\|_{L^2((1+|v|)^{-\b}Mdv)}
\\
&\le\! C\!\left(1\!+\!\eps
\left\|\frac{\sqrt{G_{\eps 1}}-1}{\eps}\right\|_{L^2(M_1(1+|v_1|^\b)dv_1)}\right)
\LA\left(\frac{\sqrt{G'_\eps G'_{\eps 1}}-\sqrt{G_\eps G_{\eps 1}}}{\eps^2}\right)^2\RA^{1/2}
\ea
$$
because of the upper bound in Grad's cut-off assumption (\ref{Grad-ctff}).

Hence, on account of Proposition \ref{UINTV-PR} and the entropy production 
estimate (d) in Proposition \ref{fluct-control} 
\begin{equation}
\label{Estm-Q2}
Q^2_\eps=O(1)_{L^2((1+|v|)^{-\b}Mdvdxdt)}+O(\eps)_{L^1_{loc}(dtdx;L^2((1+|v|)^{-\b}Mdv)))}\,.
\end{equation}
Both estimates (\ref{Estm-Q1}) and (\ref{Estm-Q2}) together with (\ref{Nrml-Sqrt-Bltz}) 
entail the control in Lemma \ref{ADVC-LM}.
\end{proof}

With Lemma \ref{ADVC-LM} at our disposal, we next proceed to the

\begin{proof}[Proof of Proposition \ref{UINT-PR}]$ $

\bigskip
\noindent
\underbar{\it Step 1.} We claim that, for $\alpha>1$,
\begin{equation}
\label{sqr2-sqrea2}
\begin{aligned}
\left(\frac{\sqrt{\eps^\a+G_\eps}-1}{\eps}\right)^2
-
\left(\frac{\sqrt{G_\eps}-1}{\eps}\right)^2
&=O(\eps^{\a-1})_{L^\infty_t(L^2_{loc}(dx;L^2(Mdv)))}
\\
&+O(\eps^{\a/2})_{L^\infty_t(L^1(Mdvdx))}\,.
\end{aligned}
\end{equation}
Indeed,
\begin{equation}
\label{sqr-sqrea}
\begin{aligned}
\left|\frac{\sqrt{\eps^\a+G_\eps}-1}{\eps}-\frac{\sqrt{G_\eps}-1}{\eps}\right|
&\le
\frac{\eps^\a \indc_{G_\eps>1/2}}{\eps(\sqrt{\eps^\a+G_\eps}+\sqrt{G_\eps})}
+
\eps^{\a/2-1}\indc_{G_\eps\le 1/2}
\\
&\le
O(\eps^{\a-1})_{L^\infty_{t,x,v}}
+\eps^{\a/2}\tfrac{\sqrt{2}}{\sqrt{2}-1}\left|\frac{\sqrt{G_\eps}-1}{\eps}\right|
\\
&=O(\eps^{\a-1})_{L^\infty_{t,x,v}}+O(\eps^{\a/2})_{L^\infty_t(L^2(Mdvdx))}\,,
\end{aligned}
\end{equation}
and we conclude with the decomposition
$$
\begin{aligned}
&\left|\left(\frac{\sqrt{\eps^\a+G_\eps}-1}{\eps}\right)^2
-
\left(\frac{\sqrt{G_\eps}-1}{\eps}\right)^2\right|
\\
&\qquad=
\left(O(\eps^{\a-1})_{L^\infty_{t,x,v}}+O(\eps^{\a/2})_{L^\infty_t(L^2(Mdvdx))}\right)
\\
&\qquad\times
\left(O(\eps^{\a-1})_{L^\infty_{t,x,v}}+O(\eps^{\a/2})_{L^\infty_t(L^2(Mdvdx))}
+
2\left|\frac{\sqrt{G_\eps}-1}{\eps}\right|\right)
\end{aligned}
$$
together with the fluctuation control (b) in Proposition \ref{fluct-control}.

\bigskip
\noindent
\underbar{\it Step 2.} Let $\g$ be a smooth truncation as in (\ref{Dfnt-gmm}), and set
$$
\phi^\de_\eps=\left(\frac{\sqrt{\eps^\a+G_\eps}-1}{\eps}\right)^2
    \g\left(\eps\de\left(\frac{\sqrt{\eps^\a+G_\eps}-1}{\eps}\right)\right)\,.
$$
We claim that, for each fixed $\de>0$,
\begin{equation}
\label{advc-phi}
(\eps\d_t+v\cdot\grad_x)\phi^\de_\eps=O\left(\frac1\de\right)_{L^1_{loc}(Mdvdxdt)}\,.
\end{equation}
Indeed,
$$
(\eps\d_t+v\cdot\grad_x)\phi^\de_\eps=
\tilde\g\left(\eps\de\left(\frac{\sqrt{\eps^\a+G_\eps}-1}{\eps}\right)\right)
\left(\frac{\sqrt{\eps^\a+G_\eps}-1}{\eps}\right)(Q^1_\eps+Q^2_\eps)
$$
where $\tilde\g(Z)=2\g(Z)+Z\g'(Z)$, while $Q^1_\eps$ and $Q^2_\eps$ are defined in 
(\ref{Dfnt-Q1Q2}).

Clearly, $\tilde\g$ has support in $[0,2]$, so that
$$
\tilde\g\left(\eps\de\left(\frac{\sqrt{\eps^\a+G_\eps}-1}{\eps}\right)\right)
\left(\frac{\sqrt{\eps^\a+G_\eps}-1}{\eps}\right)=O\left(\frac1{\eps\de}\right)_{L^\infty_{t,x,v}}\,.
$$
On the other hand, the the fluctuation control (b) in Proposition \ref{fluct-control}
and the estimate (\ref{sqr-sqrea}) imply that
$$
\tilde\g\left(\eps\de\left(\frac{\sqrt{\eps^\a+G_\eps}-1}{\eps}\right)\right)
\left(\frac{\sqrt{\eps^\a+G_\eps}-1}{\eps}\right)=O(1)_{L^\infty_t(L^2_{loc}(dx;L^2(Mdv)))}\,.
$$
Together with Lemma \ref{ADVC-LM}, these last two estimates lead to the following 
bound 
$$
\begin{aligned}
(\eps\d_t+v\cdot\grad_x)\phi^\de_\eps&=
O\left(\frac{\eps^{1-\a/2}}{\de}\right)_{L^1(Mdvdxdt)}
\\
&+O(1)_{L^2_t(L^1_{loc}(dx;L^1((1+|v|)^{-\b/2}Mdv)))}
\\
&+O\left(\frac1\de\right)_{L^1_{loc}(dtdx;L^2((1+|v|)^{-\b}Mdv))}\,.
\end{aligned}
$$
Pick then $\a\in(1,2)$; the last estimate implies that (\ref{advc-phi}) holds for each 
$\de>0$, as announced.

\bigskip
\noindent
\underbar{\it Step 3.} On the other hand, we already know from the fluctuation control 
(b) in Proposition \ref{fluct-control} and (\ref{sqr2-sqrea2}) that
\begin{equation}
\label{H1}
\phi^\de_\eps=O(1)_{L^\infty_t(L^1_{loc}(Mdvdx))}\,.
\end{equation}
Moreover
\begin{equation}
\label{H2}
\phi^\de_\eps \hbox{  is locally uniformly integrable in the $v$-variable }
\end{equation}
Indeed, for each $\varphi\in L^\infty_{x,v}\cap L^\infty_x(L^1_v)$, one has
$$
\ba
\left|\int_K\int\phi^\de_\eps\varphi Mdvdx\right|
&\le\iint\left(\frac{\sqrt{G_\eps}-1}{\eps}\right)^2|\varphi|Mdxdv
\\
&+
\int_K\int\left|\left(\frac{\sqrt{\eps^\a+G_\eps}-1}{\eps}\right)^2
-\left(\frac{\sqrt{G_\eps}-1}{\eps}\right)^2\right||\varphi|Mdvdx\,.
\ea
$$
The second term is $O(\eps^{\a-1})\|\phi\|_{L^\infty}$. Hence this term can be made
smaller than any given $\eta$ whenever $\eps<\eps_0(\eta)$. Since $\eps$ denotes 
an extracted subsequence converging to $0$, there remain only finitely many terms, 
say $N\equiv N(\eta)$ that can also be made smaller that $\eta$, this time by choosing 
$\|\phi\|_{L^\infty_x(L^1_v)}$ smaller than $c\equiv c(N(\eta),\eta)$. As for the first term,
it can be made less than $\eta$ whenever $\|\phi\|_{L^\infty_x(L^1_v)}\le c'(\eta)$, by
Proposition  \ref{UINTV-PR}. Therefore
$$
\left|\int_K\int\phi^\de_\eps\varphi Mdvdx\right|\le2\eta
	\hbox{ for each $\eps$ and $\de>0$}
$$
whenever $\|\phi\|_{L^\infty_x(L^1_v)}\le\min(c(N(\eta),\eta),c'(\eta))$, which establishes
(\ref{H2}).

Applying Theorem \ref{B1-thm} (taken from \cite{GSRInvMath}) in the Appendix below, 
we conclude from (\ref{H1}), (\ref{H2}) and (\ref{advc-phi}) that 
\begin{equation}
\label{uint-phi}
\hbox{ for each }\de>0,\quad \phi^\de_\eps
	\hbox{ is locally uniformly integrable on }\bR_+\times\bR^3\times\bR^3
\end{equation}
for the measure $Mdvdxdt$. 

\bigskip
\noindent
\underbar{\it Step 4.} But, for each $\eps,\de\in(0,1)$, one has
$$
\begin{aligned}
\left(\frac{\sqrt{\eps^\a+G_\eps}-1}{\eps}\right)^2&-\phi^\de_\eps
\\
&=
\left(\frac{\sqrt{\eps^\a+G_\eps}-1}{\eps}\right)^2\left(1-\g\!\left(\!\eps\de\!
\left(\frac{\sqrt{\eps^\a+G_\eps}-1}{\eps}\right)\!\!\right)\!\!\right)
\\
&\le
\left(\frac{\sqrt{\eps^\a+G_\eps}-1}{\eps}\right)^2\indc_{G_\eps>1/\de^2}
\\
&\le
\frac1{\eps^2}G_\eps \indc_{G_\eps>1/\de^2}\le
\frac{C}{|\ln \delta|} \frac1{\eps^2}h(G_\eps -1)\indc_{G_\eps>1/\de^2}
\end{aligned}
$$
so that
$$
\left(\frac{\sqrt{\eps^\a+G_\eps}-1}{\eps}\right)^2-\phi^\de_\eps
=O\left(\frac1{|\ln\de|}\right)_{L^\infty_t(L^1(Mdvdx))}
$$
by the fluctuation control (a) in Proposition \ref{fluct-control}. This and (\ref{uint-phi}) 
imply that
\begin{equation}
\label{uint-sqrea2}
\left(\frac{\sqrt{\eps^\a+G_\eps}-1}{\eps}\right)^2
    \hbox{ is also locally uniformly integrable on }\bR_+\times\bR^3\times\bR^3
\end{equation}
for the measure $Mdvdxdt$. 

Because of the estimate (\ref{sqr2-sqrea2}) in Step 1, we finally conclude that
\begin{equation}
\label{uint-sqr2}
\left(\frac{\sqrt{G_\eps}-1}{\eps}\right)^2
    \hbox{ is locally uniformly integrable on }\bR_+\times\bR^3\times\bR^3
\end{equation}
for the measure $Mdvdxdt$. 

Together with the control of large velocities in Proposition \ref{UINTV-PR}, the
statement (\ref{uint-sqr2}) entails Proposition \ref{UINT-PR}.
\end{proof}

\smallskip
Here is a first consequence of Proposition \ref{UINT-PR}, bearing on the relaxation 
to infinitesimal Maxwellians.

\begin{Prop}\label{RLXT-PR}
Under the assumptions of Theorem \ref{BNSW-TH}, one has
$$
\frac{\sqrt{G_\eps}-1}{\eps}-\Pi\frac{\sqrt{G_\eps}-1}{\eps} 
	\to 0\hbox{ in }{L^2_{loc}(dtdx;L^2((1+|v|)^p Mdv)}
$$
for each $p<2$ as $\eps\to 0$.
\end{Prop}

\begin{proof}
By Proposition \ref{UINT-PR}, the family
$$
(1+|v|)^p\left(\frac{\sqrt{G_\eps}-1}{\eps}-\Pi\frac{\sqrt{G_\eps}-1}{\eps}\right)^2
$$
is uniformly integrable on $[0,T]\times K\times\bR^3$ for the measure $Mdvdxdt$, 
for each $T>0$ and each compact $K\subset\bR^3$.

On the other hand, (\ref{Crcv-1}) and the the fluctuation control (b) in Proposition 
\ref{fluct-control} imply that
$$
\left(\frac{\sqrt{G_\eps}-1}{\eps}-\Pi\frac{\sqrt{G_\eps}-1}{\eps}\right)\to 0
	\hbox{ in }L^1_{loc}(Mdvdxdt)
$$
and therefore in $Mdvdxdt$-measure locally on $\bR_+\times\bR^3\times\bR^3$.

Therefore 
$$
(1+|v|)^p\left(\frac{\sqrt{G_\eps}-1}{\eps}-\Pi\frac{\sqrt{G_\eps}-1}{\eps}\right)^2\to 0
\hbox{ in }L^1_{loc}(dtdx;L^1(Mdv))\,,
$$
which implies the convergence stated above.
\end{proof}

\bigskip
We conclude this section with the following variant of the classical velocity averaging 
theorem \cite{GLPS,GPS}, stated as Theorem \ref{B2-thm} in \cite{GSRInvMath}. This result is
needed in order to handle the nonlinear terms appearing in the hydrodynamic limit. 

\begin{Prop}\label{CPTX-PR}
Under the assumptions of Theorem \ref{BNSW-TH}, for each $\xi\in L^2(Mdv)$, each
$T>0$ and each compact $K\subset\bR^3$
$$
\int_0^T\int_K|\la\xi g_\eps\g_\eps\ra(t,x+y)-\la\xi g_\eps\g_\eps\ra(t,x)|^2dxdt\to 0
$$
as $|y|\to 0^+$, uniformly in $\eps>0$.
\end{Prop}

\begin{proof}
Observe that
$$
g_\eps\g_\eps-2\frac{\sqrt{G_\eps}-1}{\eps}
=\frac{\sqrt{G_\eps}-1}{\eps}((\sqrt{G_\eps}+1)\g_\eps-2)\,;
$$
since, up to extraction,
$$
(\sqrt{G_\eps}+1)\g_\eps-2\to 0\hbox{ a.e. and }|(\sqrt{G_\eps}+1)\g_\eps-2|\le 3+\sqrt{2}\,,
$$
it follows from Proposition \ref{UINT-PR} and Theorem \ref{A1-thm} in the Appendix 
below, referred to as the Product Limit Theorem, that
\begin{equation}
\label{gga-sqrt}
g_\eps\g_\eps-2\frac{\sqrt{G_\eps}-1}{\eps}\to 0\hbox{ in }L^2_{loc}(dtdx;L^2(Mdv))
\end{equation}
as $\eps\to 0$.

This estimate, and step 1 in the proof of Proposition \ref{UINT-PR} (and especially the
estimate (\ref{sqr2-sqrea2}) there) shows that one can replace $g_\eps\g_\eps$ with 
$\frac{\sqrt{\eps^\a+G_\eps}-1}{\eps}$ with $\a>1$ in the equicontinuity statement 
of Proposition \ref{CPTX-PR}.

Using (\ref{uint-sqrea2}) shows that, for each $\a\in(1,2)$, the family 
$$
\left(\frac{\sqrt{\eps^\a+G_\eps}-1}{\eps}\right)^2
	\hbox{ is locally uniformly integrable on }\bR_+\times\bR^3\times\bR^3
$$ 
for the measure $Mdvdxdt$. In view of the estimate (\ref{sqr2-sqrea2}) and Proposition 
\ref{UINTV-PR}, we also control the contribution of large velocities in the above term, 
so that, for each $T>0$ and each compact  $K\subset \bR^3$,
$$
\left(\frac{\sqrt{\eps^\a+G_\eps}-1}{\eps}\right)^2\hbox{ is  uniformly integrable on }
[0,T] \times K\times\bR^3
$$ 
for the measure $Mdvdxdt$. 

On the other hand, Lemma \ref{ADVC-LM} shows that the family 
$$
(\eps\d_t+v\cdot\grad_x)\frac{\sqrt{\eps^\a+G_\eps}-1}{\eps}
	\hbox{ is bounded in } L^1_{loc}(Mdvdxdt)\,.
$$

Applying then Theorem \ref{B2-thm} (taken from \cite{GSRInvMath}) in the Appendix
below shows that, for each $T>0$ and each compact  $K\subset\bR^3$, one has
$$
\int_0^T\int_K\left|\La\xi\frac{\sqrt{\eps^\a+G_\eps}-1}{\eps}\Ra(t,x+y)
    -\La\xi\frac{\sqrt{\eps^\a+G_\eps}-1}{\eps}\Ra(t,x)\right|^2dxdt\to 0
$$
as $|y|\to 0$ uniformly in $\eps$, which concludes the proof of Proposition \ref{CPTX-PR}.
\end{proof}


\section{Vanishing of conservation defects}


Conservation defects appear in the renormalized form of the Boltzmann equation precisely 
because the natural symmetries of the collision integral are broken by the renormalization 
procedure. However, these conservation defects vanish in the hydrodynamic limit, as shown
by the following 

\begin{Prop}\label{DF-PR}
Under the same assumptions as in Theorem \ref{BNSW-TH}, for each 
$\xi\in\Span\{1,v_1,v_2,v_3,|v|^2\}$, one has the following convergence for the conservation 
defects $\bD_\eps(\xi)$ defined by (\ref{Dfnt-D})~:
$$
\bD_\eps(\xi)\to 0\hbox{ in }L^1_{loc}(dtdx)\hbox{ as }\eps\to 0\,.
$$
\end{Prop}

\begin{proof}
For $\xi\in\Span\{1,v_1,v_2,v_3,|v|^2\}$, the associated defect $\bD_\eps(\xi)$ is split as follows:
\begin{equation}
\label{Dcmp-D}
\bD_\eps(\xi)=\bD^1_\eps(\xi)+\bD^2_\eps(\xi)
\end{equation}
with
$$
\bD^1_\eps(\xi)=\frac1{\eps^3}
\LA\xi_{K_\eps}\hat\g_\eps\left(\sqrt{G'_\eps G'_{\eps 1}}-\sqrt{G_\eps G_{\eps 1}}\right)^2\RA\,,
$$
and
$$
\bD^2_\eps(\xi)=\frac2{\eps^3}
\LA\xi_{K_\eps}\hat\g_\eps\sqrt{G_\eps G_{\eps 1}}
\left(\sqrt{G'_\eps G'_{\eps 1}}-\sqrt{G_\eps G_{\eps 1}}\right)\RA\,,
$$
with the notation (\ref{Def-mu}) and (\ref{Not<<>>}).

\bigskip
That the term $\bD^1_\eps(\xi)$ vanishes for $\xi(v)=O(|v|^2)$ as $|v|\to+\infty$ is easily seen as
follows:
\begin{equation}
\label{D1->0}
\begin{aligned}
\|\bD^1_\eps(\xi)\|_{L^1_{t,x}}&\le\eps\|\xi_{K_\eps}\hat\g_\eps\|_{L^\infty_{t,x,v}}
\left\|\frac{\sqrt{G'_\eps G'_{\eps 1}}-\sqrt{G_\eps G_{\eps 1}}}{\eps^2}\right\|^2_{L^2_{t,x,\mu}}
\\
&\le\eps O(K_\eps)O(1)=O(\eps|\ln\eps|)\,,
\end{aligned}
\end{equation}
because of the entropy production estimate in Proposition \ref{fluct-control} (d) and the choice
of $K_\eps$ in (\ref{DefKeps}).

\bigskip
We further decompose $\bD^2_\eps(\xi)$ in the following manner:
\begin{equation}
\label{Dcmp-D2}
\bD^2_\eps(\xi)=\bD^{21}_\eps(\xi)+\bD^{22}_\eps(\xi)+\bD^{23}_\eps(\xi)
\end{equation}
with
$$
\bD^{21}_\eps(\xi)=
-\frac2{\eps}\LA\xi\indc_{|v|^2>K_\eps}\hat\g_\eps\frac{\sqrt{G'_\eps G'_{\eps 1}}
-\sqrt{G_\eps G_{\eps 1}}}{\eps^2}\sqrt{G_\eps G_{\eps 1}}\RA\,,
$$
$$
\bD^{22}_\eps(\xi)=
\frac2{\eps}\LA\xi\hat\g_\eps(1-\hat\g_{\eps 1}\hat\g'_\eps\hat\g'_{\eps 1})
\frac{\sqrt{G'_\eps G'_{\eps 1}}-\sqrt{G_\eps G_{\eps 1}}}{\eps^2}\sqrt{G_\eps G_{\eps 1}}\RA\,,
$$
and, by symmetry in the $v$ and $v_1$ variables,
$$
\bD^{23}_\eps(\xi)=
\frac1{\eps}\LA(\xi+\xi_1)\hat\g_\eps\hat\g_{\eps 1}\hat\g'_\eps\hat\g'_{\eps 1}
\frac{\sqrt{G'_\eps G'_{\eps 1}}-\sqrt{G_\eps G_{\eps 1}}}{\eps^2}\sqrt{G_\eps G_{\eps 1}}\RA\,.
$$

The terms $\bD^{21}_\eps(\xi)$ and $\bD^{23}_\eps(\xi)$ are easily mastered by the following 
classical estimate on the tail of Gaussian distributions (see for instance \cite{GSRInvMath} on 
p. 103 for a proof).

\begin{Lem}\label{GSTL-LM}
Let $\bG_N(z)$ be the centered, reduced Gaussian density in $\bR^N$, i.e.
$$
\bG_N(z)=\frac1{(2\pi)^{N/2}}e^{-\tfrac12|z|^2}\,.
$$
Then
$$
\int_{|z|^2>R}|z|^p\bG_N(z)dz\sim (2\pi)^{-N/2}|\bS^{N-1}|R^{\frac{p+N}2-1}e^{-\tfrac12R}
$$
as $R\to+\infty$.
\end{Lem}

Indeed, because of the upper bound on the collision cross-section in (\ref{Grad-ctff}), for each 
$T>0$ and each compact $K\subset \bR^3$,
$$
\begin{aligned}
\|\bD^{21}_\eps(\xi)\|&_{L^1([0,T]\times K)}
\\
\le&\frac2{\eps}
\|\xi\indc_{|v|^2>K_\eps}\hat\g_\eps\sqrt{G_\eps G_{\eps 1}}\|_{L^2([0,T]\times K,L^2_\mu)}
\left\|\frac{\sqrt{G'_\eps G'_{\eps 1}}-\sqrt{G_\eps G_{\eps 1}}}{\eps^2}\right\|_{L^2_{t,x,\mu}}
\\
\le&\frac{C_b^{1/2}}{\eps}\|\xi^2\indc_{|v|^2>K_\eps}(1+|v|)^{\b}\|_{L^1(Mdv)}^{1/2}
\|\hat\g_\eps\sqrt{G_\eps}\|_{L^\infty_{t,x,v}}\\
&\quad\|G_{\eps 1}(1+|v_1|)^\b\|^{1/2}_{L^1([0,T]\times K,L^1(M_1dv_1))}
 \left\|\frac{\sqrt{G'_\eps G'_{\eps 1}}-\sqrt{G_\eps G_{\eps 1}}}{\eps^2}\right\|_{L^2_{t,x,\mu}}
\end{aligned}
$$
In the last right-hand side of the above chain of inequalities, one has obviously
$$
\|\hat\g_\eps\sqrt{G_\eps}\|_{L^\infty_{t,x,v}}=O(1)\,.$$
From Young's inequality and the entropy bound (\ref{entropy-bound}), we deduce that
$$ 
\begin{aligned}
G_\eps(1+|v|)^2&\leq (1+|v|^2) +4\left( h(G_\eps-1) + h^* \left(\frac{1+|v|^2}{4}\right)\right)\\
&=O(1)_{L^1([0,T]\times K,L^1(Mdv))}\,.
\end{aligned}
$$
Lemma \ref{GSTL-LM} and the condition $\xi(v)=O(|v|^2)$ as $|v|\to+\infty$ imply that
$$
\|\xi^2\indc_{|v|^2>K_\eps}(1+|v|)^{\b}\|_{L^1(Mdv)}^{1/2}=O(K_\eps^{\frac{\b+5}2}e^{-\tfrac12K_\eps})
=O(\eps^{\mathbf{k}/2}|\ln\eps|^{\frac{\b+5}2})\,,
$$
on account of (\ref{DefKeps}). Thus
\begin{equation}
\label{D21->0}
\|\bD^{21}_\eps(\xi)\|_{L^1([0,T]\times K)}=O(\eps^{\mathbf{k}/2-1}|\ln\eps|^{\frac{\b+5}2})\to 0
\end{equation}
for all $\xi(v)=O(|v|^2)$ as $|v|\to+\infty$ as soon as $\mathbf{k}>2$.

\bigskip
Next we handle $\bD^{23}_\eps(\xi)$. Whenever $\xi$ is a collision invariant (i.e. whenever $\xi$ 
belongs to the linear span of $\{1,v_1,v_2,v_3,|v|^2\}$) then $\xi+\xi_1=\xi'+\xi'_1$, and using the 
$(v,v_1)-(v',v'_1)$ symmetry (\ref{CollSym}) in the integral defining $\bD^{23}_\eps(\xi)$ leads to
$$
\begin{aligned}
\bD^{23}_\eps(\xi)&=
-\frac1{\eps}\LA(\xi+\xi_1)\hat\g_\eps\hat\g_{\eps 1}\hat\g'_\eps\hat\g'_{\eps 1}
\frac{(\sqrt{G'_\eps G'_{\eps 1}}-\sqrt{G_\eps G_{\eps 1}})^2}{2\eps^2}\RA
\\
&=-\bD^{231}_\eps(\xi)-\bD^{232}_\eps(\xi)\,,
\end{aligned}
$$
where
$$
\bD^{231}_\eps(\xi)=
\tfrac12 \eps \LA(\xi+\xi_1)\indc_{|v|^2+|v_1^2|\le K_\eps}
\hat\g_\eps\hat\g_{\eps 1}\hat\g'_\eps\hat\g'_{\eps 1}
\frac{(\sqrt{G'_\eps G'_{\eps 1}}-\sqrt{G_\eps G_{\eps 1}})^2}{\eps^4}\RA\,,
$$
and
$$
\bD^{232}_\eps(\xi)=
\tfrac12 \eps\LA(\xi+\xi_1)\indc_{|v|^2+|v_1^2|>K_\eps}
\hat\g_\eps\hat\g_{\eps 1}\hat\g'_\eps\hat\g'_{\eps 1}
\frac{(\sqrt{G'_\eps G'_{\eps 1}}-\sqrt{G_\eps G_{\eps 1}})^2}{\eps^4}\RA\,.
$$
Then
$$\begin{aligned}
\|\bD^{231}_\eps(\xi)\|_{L^1_{t,x}}
&\le\eps\left\|
\frac{\sqrt{G'_\eps G'_{\eps 1}}-\sqrt{G_\eps G_{\eps 1}}}{\eps^2}\right\|^2_{L^2_{t,x,\mu}}
\\
&\times\tfrac12\|(\xi+\xi_1)\indc_{|v|^2+|v_1|^2\le K_\eps}
\hat\g_\eps\hat\g_{\eps 1}\hat\g'_\eps\hat\g'_{\eps 1}\|_{L^\infty_{t,x,v,v_1,\om}}
\\
&=\eps\cdot O(1)\cdot O(K_\eps)\|\hat\g\|^4_{L^\infty}
\end{aligned}
$$
so that
\begin{equation}
\label{D231->0}
\|\bD^{231}_\eps(\xi)\|_{L^1_{t,x}}=O(\eps K_\eps)\to 0 \hbox{ as }\eps \to 0\,.
\end{equation}

On the other hand, since $G_\eps\in[0,2]$ whenever $\hat\g(G_\eps)\not=0$,
$$\begin{aligned}
{}&\|\bD^{232}_\eps(\xi)\|_{L^\infty_{t,x}}\le 
16\|\hat\g\|^4_{L^\infty}\frac1{\eps^3}\|\tfrac12(\xi+\xi_1)\indc_{|v|^2+|v_1|^2>K_\eps}\|_{L^1_\mu}
\\
&\le O\left(\frac1{\eps^3}\right)
\|(1+|v|^2+|v_1|^2)(1+|v-v_1|)^\b\indc_{|v|^2+|v_1|^2>K_\eps}\|_{L^1(MM_1dvdv_1)}
\\
&=O\left(\frac1{\eps^3}\right)
\|(1+|v|^2+|v_1|^2)^{1+\b/2}\indc_{|v|^2+|v_1|^2>K_\eps}\|_{L^1(MM_1dvdv_1)}
\\
&=O\left(\frac1{\eps^3}\right)O\left(e^{-K_\eps/2}K_\eps^{\frac{\b+6}2}\right)
\end{aligned}
$$
so that
\begin{equation}
\label{D232->0}
\|\bD^{232}_\eps(\xi)\|_{L^\infty_{t,x}}=O\left(\eps^{{\bf k}/2-3}|\ln\eps|^{\frac{\b+6}2}\right)
\to 0\hbox{ as }\eps \to 0\,,
\end{equation}
for ${\bf k}>6$, by a direct application of Lemma \ref{GSTL-LM} in $\bR^3_{v}\times\bR^3_{v_1}$ 
--- i.e. with $N=6$.

\bigskip
Whereas the terms $\bD^1_\eps(\xi)$, $\bD^{21}_\eps(\xi)$, and $\bD^{23}_\eps(\xi)$ are shown
to vanish by means of only the entropy and entropy production bounds in Proposition \ref{fluct-control}
(a)-(d) and Lemma \ref{GSTL-LM}, the term $\bD^{22}_\eps(\xi)$ is much less elementary to handle.

First, we split $\bD^{22}_\eps(\xi)$ as
$$
\begin{aligned}
{}&\bD^{22}_\eps(\xi)=\frac2{\eps}\LA\xi\hat\g_\eps(1-\hat\g_{\eps 1})\sqrt{G_\eps G_{\eps 1}}
\frac{\sqrt{G'_\eps G'_{\eps 1}}-\sqrt{G_\eps G_{\eps 1}}}{\eps^2}\RA
\\
&+
\frac2{\eps}\LA\xi(\hat\g_\eps\hat\g_{\eps 1}(1-\hat\g'_\eps)
+\hat\g_\eps\hat\g_{\eps 1}\hat\g'_\eps(1-\hat\g'_{\eps 1}))
\frac{\sqrt{G'_\eps G'_{\eps 1}}-\sqrt{G_\eps G_{\eps 1}}}{\eps^2}\sqrt{G_\eps G_{\eps 1}}\RA
\\
&\quad\quad\quad=\bD^{221}_\eps(\xi)+\bD^{222}_\eps(\xi)\,.
\end{aligned}
$$

For each $T>0$ and each compact $K\subset\bR^3$, the first term satisfies
$$
\begin{aligned}
\,&\|\bD^{221}_\eps(\xi)\|_{L^1([0,T]\times K)}
\\
&\le 2C_b\left\|\frac1\eps(1-\hat\g_{\eps 1})\sqrt{G_{\eps 1}}
(1+|v_1|)^{\b/2}\right\|_{L^2([0,T]\times K;L^2(M_1dv_1)))}\|\hat\g_\eps\sqrt{G_\eps}\|_{L^\infty_{t,x,v}}
\\
&
\qquad \||\xi|(1+|v|)^{\b/2}\|_{L^2(Mdv)}\left\|\frac{\sqrt{G'_\eps G'_{\eps 1}}-\sqrt{G_\eps G_{\eps 1}}}{\eps^2}\right\|_{L^2_{t,x,\mu}}
\\
&=O(1)
\left\|\frac1\eps(1-\hat\g_{\eps 1})\sqrt{G_{\eps 1}}
(1+|v_1|)^{\b/2}\right\|_{L^2([0,T]\times K;L^2(M_1dv_1)))}
\end{aligned}
$$
provided that $\xi(v)=O(|v|^m)$ for some $m\in\bN$. 

Since $\Supp(1-\hat\g)\subset[\tfrac32,+\infty)$, then 
$\frac{\sqrt{G_\eps}}{\sqrt{G_\eps}-1}\le\frac{\sqrt{3/2}}{\sqrt{3/2}-1}$ whenever $\hat\g_\eps\not=1$,
and one has
$$
\frac1\eps|1-\hat\g_\eps|\sqrt{G_\eps}\le\tfrac{\sqrt{3}}{\sqrt{3}-\sqrt{2}}
|1-\hat\g_\eps|\frac{|\sqrt{G_\eps}-1|}{\eps}\,.
$$
Furthermore, as
 \begin{equation}
\label{1-hg->0}
|1-\hat\g_\eps|\le 1+\|\g'\|_{L^\infty}\hbox{ and }1-\hat\g_\eps\to 0\hbox{ a.e.,}
\end{equation}
the uniform integrability stated in Proposition \ref{UINT-PR} and the Product Limit Theorem (see
Appendix A) imply that
\begin{equation}
\label{1-gamma}
|1-\hat\g_\eps|\frac{|\sqrt{G_\eps}-1|}{\eps} \to 0 \hbox{ in } L^2([0,T]\times K, L^2(M(1+|v|)^\b dv)).
\end{equation}

Thus
\begin{equation}
\label{D221->0}
\|\bD^{221}_\eps(\xi)\|_{L^1([0,T]\times K)}
\to 0\hbox{ as }\eps\to 0\,.
\end{equation}

\bigskip
Finally, we consider the term $\bD^{222}_\eps(\xi)$: one has
$$
\begin{aligned}
\|\bD^{222}_\eps(\xi)\|&_{L^1([0,T]\times K)}\\
&\le \frac2\eps\left(\|(1-\hat\g'_\eps)\xi\|_{L^2([0,T]\times K;L^2_\mu)}
+\|(1-\hat\g'_{\eps_1})\xi\|_{L^2([0,T]\times K;L^2_\mu)}\right)
\\
&\qquad   \|\hat\g_\eps\sqrt{G_\eps}\|^2_{L^\infty_{t,x,v}}\left\|\frac{\sqrt{G'_\eps G'_{\eps 1}}
-\sqrt{G_\eps G_{\eps 1}}}{\eps^2}\right\|_{L^2_{t,x,\mu}}
\\
&= O(1)\left\|\frac{1-\hat\g_\eps}{\eps}
(1+|v|^2+|v_1|^2)\right\|_{L^2([0,T]\times K;L^2_\mu)}
\\
&= O(1)\left\|\frac{1-\hat\g_\eps}{\eps}(1+|v|)^{2+\b/2}\right\|_{L^2([0,T]\times K;L^2(Mdv))}\,,
\end{aligned}
$$
where the first equality uses the $(vv_1)-(v'v'_1)$ symmetry in (\ref{CollSym}).

Since $\Supp(1-\hat\g)\subset[\tfrac32,+\infty)$, $\frac1{\sqrt{G_\eps}-1}\le\frac{1}{\sqrt{3/2}-1}$ 
whenever $\hat\g_\eps\not=1$, one has
$$
\begin{aligned}
&\frac{|1-\hat\g_\eps|^2}{\eps^2}\leq\tfrac{\sqrt{2}}{\sqrt{3}-\sqrt{2}}\frac{|1-\hat\g_\eps|}{\eps}
\frac{\sqrt{G_\eps}-1}{\eps}
\\
&\leq\tfrac{\sqrt{2}}{\sqrt{3}-\sqrt{2}}\frac{|1-\hat\g_\eps|}{\eps}
\left( \Pi\frac{\sqrt{G_\eps}-1}{\eps}
	+ \left(\frac{\sqrt{G_\eps}-1}{\eps}-\Pi\frac{\sqrt{G_\eps}-1}{\eps}\right)\right)
\end{aligned}
$$
By (\ref{1-hg->0}) and (\ref{1-gamma})
\be
\lb{1-hg/eps}
\frac{|1-\hat\g_\eps|}{\eps} \leq\frac{1+\| \gamma'\|_{L^\infty}}{\eps} 
\hbox{ and } \frac{|1-\hat\g_\eps|}{\eps} \to 0 \hbox{ in } L^2_{loc}(dtdx, L^2(Mdv))
\ee
since $\sqrt{G_\eps}-1>\sqrt{3/2}-1$ whenever $\hat\g_\eps\not=1$, whereas by Proposition
\ref{fluct-control} (b) and Lemma \ref{I-PI-LM}
$$
\begin{aligned}
\Pi\frac{\sqrt{G_\eps}-1}{\eps} =O(1)_{L^\infty_t(L^2_x(L^q(Mdv)))} \,,\\
\frac{\sqrt{G_\eps}-1}{\eps}-\Pi\frac{\sqrt{G_\eps}-1}{\eps} =O(\eps )_{L^1_{loc}(dtdx, L^2(Mdv))}\,,
\end{aligned}
$$
for all $q>+\infty$.
Then,
$$
\frac{|1-\hat\g_\eps|^2}{\eps^2} =O(1)_{L^1_{loc}(dtdx, L^q(Mdv))}
$$
for all $q<2$. In particular, for each $r<+\infty$, $\left(\frac{1}{\eps}(1-\hat\g_\eps)(1+|v|)^{r}\right)$ 
is uniformly bounded in $L^2_{loc}(dtdx, L^2(Mdv))$. By interpolation with (\ref{1-hg/eps}) we 
conclude that
\begin{equation}
\label{g-1/eps}
\left\|\frac{1-\hat\g_\eps}{\eps}(1+|v|)^{r}\right\|^2_{L^2([0,T]\times K;L^2(Mdv))}\to 0 
\hbox{ as }\eps \to 0
\end{equation}
and consequently
\begin{equation}
\label{D222->0}
\bD^{222}_\eps(\xi)\to 0\hbox{ in }L^1_{loc}(dtdx)\hbox{ as }\eps\to 0\,.
\end{equation}

The convergences (\ref{D1->0}), (\ref{D21->0}), (\ref{D221->0}), (\ref{D222->0}), (\ref{D231->0})
and (\ref{D232->0}) eventually imply Proposition \ref{DF-PR}.
\end{proof}

\smallskip
\noindent
\textbf{Remark.} The same arguments leading to (\ref{1-gamma}) and to (\ref{g-1/eps}) imply 
that, for each $r\in\bR$,
\begin{equation}
\label{ga-1/eps}
\left\|\frac{1-\g_\eps}{\eps}(1+|v|)^{r}\right\|^2_{L^2([0,T]\times K;L^2(Mdv))}
	\to 0 
\hbox{ as }\eps \to 0\,.
\end{equation}


\section{Asymptotic behavior of the flux terms}


The purpose of the present section is to establish the following

\begin{Prop}\label{ASMP-PR}
Under the same assumptions as in Theorem \ref{BNSW-TH}, one has
$$
\bF_\eps(\zeta)-2\La\zeta\left(\Pi\frac{\sqrt{G_\eps}-1}{\eps}\right)^2\Ra
    +\frac2{\eps^2}\La\hat\zeta\cQ(\sqrt{G_\eps},\sqrt{G_\eps})\Ra\to 0
    \hbox{ in }L^1_{loc}(dtdx)
$$
as $\eps\to 0$, where $\zeta$ and $\hat\zeta$ designate respectively either $A$ and $\hat A$
or $B$ and $\hat B$ defined by (\ref{Def-AB}) and (\ref{LABhat}).
\end{Prop}

\begin{proof}
First, we decompose the flux term $\bF_\eps(\zeta)$ as follows:
$$
\begin{aligned}
\bF_\eps(\zeta)&=\frac1\eps\La\zeta_{K_\eps}g_\eps\g_\eps\Ra
    =\La\zeta_{K_\eps}\frac{G_\eps-1}{\eps^2}\g_\eps\Ra
\\
&=\La\zeta_{K_\eps}\left(\frac{\sqrt{G_\eps}-1}{\eps}\right)^2\g_\eps\Ra
+
\frac2\eps\La\zeta_{K_\eps}\frac{\sqrt{G_\eps}-1}{\eps}\g_\eps\Ra
\\
&=\bF^1_\eps(\zeta)+\bF^2_\eps(\zeta)\,.
\end{aligned}
$$
We further split the term $\bF^1_\eps(\zeta)$ as
$$
\bF^1_\eps(\zeta)=\bF^{11}_\eps(\zeta)+\bF^{12}_\eps(\zeta)+\bF^{13}_\eps(\zeta)
$$
with
\begin{equation}
\label{Dcmp-F1}
\begin{aligned}
\bF^{11}_\eps(\zeta)&=
\La\zeta_{K_\eps}\left(\frac{\sqrt{G_\eps}-1}{\eps}-\Pi\frac{\sqrt{G_\eps}-1}{\eps}\right)
    \left(\frac{\sqrt{G_\eps}-1}{\eps}+\Pi\frac{\sqrt{G_\eps}-1}{\eps}\right)\g_\eps\Ra\,,
\\
\bF^{12}_\eps(\zeta)&=
\La\zeta(\indc_{|v|^2\le K_\eps}\g_\eps-1)\left(\Pi\frac{\sqrt{G_\eps}-1}{\eps}\right)^2\Ra\,,
\\
\bF^{13}_\eps(\zeta)&=\La\zeta\left(\Pi\frac{\sqrt{G_\eps}-1}{\eps}\right)^2\Ra\,.
\end{aligned}
\end{equation}

\bigskip
The term $\bF^{12}_\eps(\zeta)$ is easily disposed of. Indeed, the definition (\ref{Dfnt-Pi}) 
of the hydrodynamic projection $\Pi$ implies that $\left(\Pi\frac{\sqrt{G_\eps}-1}{\eps}\right)^2(1+|v|)^p$ 
is, for each $p\ge 0$, a (finite) linear combination of functions of $v$ of order $O(|v|^{p+4})$ 
as $|v|\to+\infty$, with coefficients that are quadratic in $\la\xi\frac{\sqrt{G_\eps}-1}{\eps}\ra$ 
for $\xi\in\{1,v_1,v_2,v_3,|v|^2\}$. Together with Proposition \ref{UINT-PR}, this implies that, 
for each $T>0$ and each compact $K\subset\bR^3$,
\begin{equation}
\label{Uintv-Pi}
\left(\Pi\frac{\sqrt{G_\eps}-1}{\eps}\right)^2(1+|v|)^p\hbox{ is uniformly integrable on }
[0,T]\times K\times\bR^3
\end{equation}
for the measure $Mdvdxdt$. On the other hand,
$$
\indc_{|v|^2\le K_\eps}\g_\eps-1\to 0
    \hbox{ and }|\indc_{|v|^2\le K_\eps}\g_\eps-1|\le 1\hbox{ a.e.}\,.
$$
Since $\zeta(v)=O(|v|^3)$ as $|v|\to+\infty$, this and the Product Limit Theorem imply that
\begin{equation}
\label{Estm-F12}
\bF^{12}_\eps(\zeta)\to 0\hbox{ in }L^1_{loc}(dtdx)\,.
\end{equation}

\bigskip
The term $\bF^{11}_\eps(\zeta)$ requires a slightly more involved treatment. We start with 
the following decomposition: for each $T>0$ and each compact $K\subset \bR^3$,
\begin{equation}
\label{Estm-F11-1}
\begin{aligned}
\|\bF^{11}_\eps(\zeta)\|_{L^1([0,T]\times K)}
\le\left\|\zeta_{K_\eps}\g_\eps\left(\frac{\sqrt{G_\eps}-1}{\eps}
    +\Pi\frac{\sqrt{G_\eps}-1}{\eps}\right)\right\|_{L^2([0,T]\times K;L^2(Mdv))}
\\
\times\left\|\frac{\sqrt{G_\eps}-1}{\eps}
    -\Pi\frac{\sqrt{G_\eps}-1}{\eps}\right\|_{L^2([0,T]\times K;L^2(Mdv))}
\end{aligned}
\end{equation}
Since $\g_\eps=\g(G_\eps)=0$ whenever $G_\eps>2$, one has for each $q<+\infty$,
$$
\begin{aligned}
\g_\eps&\left(\frac{\sqrt{G_\eps}-1}{\eps}\right)^2
\\
&=\g_\eps\left(\frac{\sqrt{G_\eps}-1}{\eps}\right)\left(\Pi \frac{\sqrt{G_\eps}-1}{\eps}
+\left( \frac{\sqrt{G_\eps}-1}{\eps}-\Pi \frac{\sqrt{G_\eps}-1}{\eps}\right)\right)
\\
&= O (1)_{L^\infty_t(L^2(dxMdv))} O(1)_{L^\infty_t(L^2_x(L^q(Mdv)))} 
	+O\left(\frac1\eps\right) O(\eps )_{L^1_{loc}(dtdx, L^2(Mdv))}
\end{aligned}$$
In particular
$$
\left\|\zeta_{K_\eps}\g_\eps\frac{\sqrt{G_\eps}-1}{\eps}\right\|_{L^2([0,T]\times K;L^2(Mdv))}=O(1)
$$
since $\zeta(v)=O(|v|^3)$ as $|v|\to+\infty$. This and (\ref{Uintv-Pi}) imply that
\begin{equation}
\label{Estm-F11-2}
\left\|\zeta_{K_\eps}\g_\eps\left(\frac{\sqrt{G_\eps}-1}{\eps}
+\Pi\frac{\sqrt{G_\eps}-1}{\eps}\right)\right\|_{L^2_{loc}(dtdx;L^2(Mdv))}=O(1)\,.
\end{equation}
Using (\ref{Estm-F11-1}), (\ref{Estm-F11-2}) and Proposition \ref{RLXT-PR} shows that
\begin{equation}
\label{Estm-F11}
\bF^{11}_\eps(\zeta)\to 0\hbox{ in }L^1_{loc}(dtdx)\,.
\end{equation}
This and (\ref{Estm-F12}) imply that
\begin{equation}
\label{F1->}
\bF^1_\eps(\zeta)-\La\zeta\left(\Pi\frac{\sqrt{G_\eps}-1}{\eps}\right)^2\Ra\to 0
\hbox{ in }L^1_{loc}(dtdx)
\end{equation}
as $\eps\to 0$.

\bigskip
Next we handle the term $\bF^2_\eps(\zeta)$. We first decompose it as follows:
\begin{equation}
\label{Dcmp-F2}
\begin{aligned}
\bF^2_\eps(\zeta)&=-\frac2\eps\La\zeta\indc_{|v|^2>K_\eps}\g_\eps
\frac{\sqrt{G_\eps}-1}{\eps}\Ra
\\
&+
2\La\zeta\frac{\g_\eps-1}{\eps}\frac{\sqrt{G_\eps}-1}{\eps}\Ra
+
\frac2\eps\La\zeta\frac{\sqrt{G_\eps}-1}{\eps}\Ra
\\
&=\bF^{21}_\eps(\zeta)+\bF^{22}_\eps(\zeta)+\bF^{23}_\eps(\zeta)\,.
\end{aligned}
\end{equation}
Then, by (\ref{Entr-estm2}) and Lemma \ref{GSTL-LM}, one has
\begin{equation}
\label{Estm-F21}
\begin{aligned}
\|\bF^{21}_\eps(\zeta)\|_{L^\infty_t(L^2_x)}
&\le\frac2\eps\|\g\|_{L^\infty}\|\zeta\indc_{|v|^2>K_\eps}\|_{L^2(Mdv)}
    \left\|\frac{\sqrt{G_\eps}-1}{\eps}\right\|_{L^\infty_t(L^2(Mdvdx))}
\\
&\le \frac2\eps O(e^{-K_\eps/2}K_\eps^2)=O(\eps^{{\bf k}/2-1}|\ln\eps|^2)\,.
\end{aligned}
\end{equation}
On the other hand, for each $T>0$ and each compact $K\subset \bR^3$,
\begin{equation}
\label{Estm-F22}
\begin{aligned}
\|\bF^{22}_\eps(\zeta)\|&_{L^1([0,T]\times K)}
\\
&\le 2T^{1/2}
\left\|\zeta\frac{\g_\eps-1}{\eps}\right\|_{L^2([0,T]\times K;L^2(Mdv))}
\left\|\frac{\sqrt{G_\eps}-1}{\eps}\right\|_{L^\infty_t(L^2(Mdvdx))}
\\
&\to 0\hbox{ as }\eps\to 0
\end{aligned}
\end{equation}
because of (\ref{Entr-estm2}) and of (\ref{ga-1/eps}), since $\zeta(v)=O(|v|^3)$ as $|v|\to+\infty$.

Finally, we transform $\bF^{23}_\eps(\zeta)$ as follows:
$$
\begin{aligned}
\bF^{23}_\eps(\zeta)&=
2\La\hat\zeta\frac1\eps\cL\left(\frac{\sqrt{G_\eps}-1}{\eps}\right)\Ra
\\
&=2\La\hat\zeta
\left(\cQ\left(\frac{\sqrt{G_\eps}-1}{\eps},\frac{\sqrt{G_\eps}-1}{\eps}\right)
	-\frac1{\eps^2}\cQ(\sqrt{G_\eps},\sqrt{G_\eps})\right)\Ra
\end{aligned}
$$
Writing
$$
\begin{aligned}
\cQ\left(\frac{\sqrt{G_\eps}-1}{\eps},\frac{\sqrt{G_\eps}-1}{\eps}\right)
=
\cQ\left(\Pi\frac{\sqrt{G_\eps}-1}{\eps},\Pi\frac{\sqrt{G_\eps}-1}{\eps}\right)
\\
+\cQ\left(\frac{\sqrt{G_\eps}-1}{\eps}-\Pi\frac{\sqrt{G_\eps}-1}{\eps},
\frac{\sqrt{G_\eps}-1}{\eps}+\Pi\frac{\sqrt{G_\eps}-1}{\eps}\right)
\end{aligned}
$$
and using the classical relation (see \cite{BGL1} for instance)
$$
\cQ(\phi,\phi)=\tfrac12\cL(\phi^2)\hbox{ for each }\phi\in\Ker\cL\,,
$$
we arrive at
$$
\begin{aligned}
{}&\cQ\left(\frac{\sqrt{G_\eps}-1}{\eps},\frac{\sqrt{G_\eps}-1}{\eps}\right)
=
\tfrac12\cL\left(\left(\Pi\frac{\sqrt{G_\eps}-1}{\eps}\right)^2\right)
\\
&\qquad\qquad+\cQ\left(\frac{\sqrt{G_\eps}-1}{\eps}-\Pi\frac{\sqrt{G_\eps}-1}{\eps},
\frac{\sqrt{G_\eps}-1}{\eps}+\Pi\frac{\sqrt{G_\eps}-1}{\eps}\right)
\end{aligned}
$$
Thus
\begin{equation}
\label{Dcmp-F23}
\begin{aligned}
\bF^{23}_\eps(\zeta)&=\La\zeta\left(\Pi\frac{\sqrt{G_\eps}-1}{\eps}\right)^2\Ra
-\frac2{\eps^2}\La\hat\zeta\cQ(\sqrt{G_\eps},\sqrt{G_\eps})\Ra
\\
&+2
\La\hat\zeta\cQ\left(\frac{\sqrt{G_\eps}-1}{\eps}-\Pi\frac{\sqrt{G_\eps}-1}{\eps},
\frac{\sqrt{G_\eps}-1}{\eps}+\Pi\frac{\sqrt{G_\eps}-1}{\eps}\right)\Ra
\end{aligned}
\end{equation}
By continuity of $\cQ$ (see (\ref{Cnt-Q})),
$$
\begin{aligned}
\left\|\La\hat\zeta\cQ\left(\frac{\sqrt{G_\eps}-1}{\eps}-\Pi\frac{\sqrt{G_\eps}-1}{\eps},
\frac{\sqrt{G_\eps}-1}{\eps}+\Pi\frac{\sqrt{G_\eps}-1}{\eps}\right)\Ra\right\|_{L^1([0,T]\times K)}
\\
\le C
\|\hat\zeta\|_{L^2(a Mdv)}
\left\|\frac{\sqrt{G_\eps}-1}{\eps}-\Pi\frac{\sqrt{G_\eps}-1}{\eps}\right\|
    _{L^2([0,T]\times K;L^2((1+|v|)^\b Mdv))}
\\
\times\left\|\frac{\sqrt{G_\eps}-1}{\eps}+\Pi\frac{\sqrt{G_\eps}-1}{\eps}\right\|
    _{L^2([0,T]\times K;L^2((1+|v|)^\b Mdv))}\to 0
\end{aligned}
$$
as $\eps\to 0$, for each $T>0$ and each compact $K\subset \bR^3$, because of (\ref{Uintv-Pi}) 
and Proposition \ref{RLXT-PR}.

\bigskip
Thus, by (\ref{Estm-F21}), (\ref{Estm-F22}) and (\ref{Dcmp-F23})
\begin{equation}
\label{F2->}
\bF^2_\eps(\zeta)-\La\zeta\left(\Pi\frac{\sqrt{G_\eps}-1}{\eps}\right)^2\Ra
    +\frac2{\eps^2}\La\hat\zeta\cQ(\sqrt{G_\eps},\sqrt{G_\eps})\Ra\to 0\,.
\end{equation}
in $L^1_{loc}(dtdx)$ as $\eps\to 0$.

The convergences (\ref{F1->}) and (\ref{F2->}) eventually imply Proposition \ref{ASMP-PR}.
\end{proof}


\section{Proof of Theorem \ref{BNSW-TH}}


Throughout this section $U\equiv U(x)$ designates an arbitrary compactly supported, $C^\infty$, 
divergence-free vector field on $\bR^3$. Taking the inner product with $U$ of both sides of
(\ref{Mmnt-v}) gives
\begin{equation}
\label{Mmnt-Pv->0}
\begin{aligned}
\d_t\int\la v_{K_{\eps}}g_\eps\g_\eps\ra\cdot Udx
-\int\bF_\eps(A):\grad_xUdx&=\int\bD_\eps(v)\cdot Udx
\\
&\to 0\hbox{ in }L^1_{loc}(dt)\,,
\end{aligned}
\end{equation}
by Proposition \ref{DF-PR}. Likewise, the energy equation (\ref{Mmnt-v2}) and Proposition 
\ref{DF-PR} lead to
\begin{equation}
\label{Mmnt-v2->0}
\begin{aligned}
\d_t\la\tfrac12(|v|^2-5)_{K_{\eps}}g_\eps\g_\eps\ra+\Div_x\bF_\eps(B)&=\bD_\eps(\tfrac12(|v|^2-5))
\\
&\to 0\hbox{ in }L^1_{loc}(dtdx)\,.
\end{aligned}
\end{equation}

By Proposition \ref{ASMP-PR}, one can decompose the fluxes as
\begin{equation}
\label{FlxDcmp}
\begin{aligned}
\bF_\eps(A)=\bFc(A)+\bFd(A)+o(1)_{L^1_{loc}(dtdx)}
\\
\bF_\eps(B)=\bFc(B)+\bFd(B)+o(1)_{L^1_{loc}(dtdx)}
\end{aligned}
\end{equation}
where
\begin{equation}
\label{DfntFcFdA}
\begin{aligned}
\bFc(A)&=2\La A\left(\Pi\frac{\sqrt{G_\eps}-1}{\eps}\right)^2\Ra\,,
\\
\bFd(A)&=-2\La\hat A\frac1{\eps^2}\cQ(\sqrt{G_\eps},\sqrt{G_\eps})\Ra\,,
\end{aligned}
\end{equation}
while
\begin{equation}
\label{DfntFcFdB}
\begin{aligned}
\bFc(B)&=2\La B\left(\Pi\frac{\sqrt{G_\eps}-1}{\eps}\right)^2\Ra
\\
\bFd(B)&=-2\La\hat B\frac1{\eps^2}\cQ(\sqrt{G_\eps},\sqrt{G_\eps})\Ra\,.
\end{aligned}
\end{equation}
Classical computations (that can be found for instance in \cite{BGL2}) using the fact 
that $A$ is orthogonal in $L^2(Mdv)$ to $\Ker\cL$ as well as to odd functions  of $v$ 
and functions of $|v|^2$ show that
$$
\bFc(A)=2\la A\otimes A\ra : \La v\frac{\sqrt{G_\eps}-1}{\eps}\Ra^{\otimes 2}\,.
$$
In a similar way, $B$ is orthogonal in $L^2(Mdv)$ to $\Ker \cL$ and to even 
functions of $v$, so that
$$
\bFc(B)=2\la B\otimes B\ra \cdot\La v\frac{\sqrt{G_\eps}-1}{\eps}\Ra
    \La(\tfrac13|v|^2-1)\frac{\sqrt{G_\eps}-1}{\eps}\Ra\,.
$$

\subsection{Convergence of the diffusion terms}\label{diffusion-cv}

The convergence of $\bFd (A)$ and $\bFd(B)$ comes only from weak compactness results, 
and from the following characterization of the weak limits.

\begin{Prop}\label{TQ=-LM}
Under the same assumptions as in Theorem \ref{BNSW-TH}, one has, up to extraction 
of a subsequence 
$\eps_n \to 0$,
\begin{equation}
\label{Cnvg-gtq}
g_{\eps_n}\to g\,,\hbox{ and }
\frac{\sqrt{G'_{\eps_n}G'_{\eps_n 1}}-\sqrt{G_{\eps_n}G_{\eps_n 1}}}{\eps_n^2}
\to \tilde q
\end{equation}
in $\wL^1_{loc}(dtdx;L^1(Mdv))$ and in $\wL^2(dtdxd\mu)$ respectively.

Furthermore $g\in L^\infty_t (L^2(dxMdv))$ is an infinitesimal Maxwellian of the form
\begin{equation}
\label{g=}
g(t,x,v)=u(t,x)\cdot v+\th(t,x)\tfrac12(|v|^2-5)\,,\quad\Div_xu=0 \,,
\end{equation}
and $\tilde q \in L^2(dtdxd\mu)$ satisfies
\begin{equation}
\label{q-lim}
\iint\tilde qb(v-v_1,\om)d\om M_1dv_1=\tfrac12 v\cdot\grad_xg
	=\tfrac12(A:\grad_xu+B\cdot\grad_x\th)\,.
\end{equation}
\end{Prop}

\begin{proof}
Proposition \ref{fluct-control} (c) shows that  
$$
(g_\eps) \hbox{ is relatively compact in }\wL^1_{loc}(dtdx;L^1(Mdv))
$$
while (\ref{Entr-prd1}) implies that
$$
\frac{\sqrt{G'_{\eps}G'_{\eps 1}}-\sqrt{G_{\eps}G_{\eps 1}}}{\eps^2}\hbox{ is relatively  compact in }
\wL^2(dtdxd\mu)\,,
$$
Pick then any sequence $\eps_n\to 0$ such that
$$g_{\eps_n}\to g\,,\hbox{ and }
\frac{\sqrt{G'_{\eps_n}G'_{\eps_n 1}}-\sqrt{G_{\eps_n}G_{\eps_n 1}}}{\eps_n^2}
\to \tilde q
$$
in $\wL^1_{loc}(dtdx;L^1(Mdv))$ and in $\wL^2(dtdxd\mu)$ respectively.

\bigskip
\noindent
\underbar{\it Step 1}~: from
(\ref{identity1}) we deduce that
$$ 
\frac1{\eps_n} (\sqrt{G_{\eps_n}}-1) \to g \hbox{ in }\wL^2_{loc}(dt, L^2(dxMdv))\,.
$$
In particular, by Proposition \ref{RLXT-PR},  $g$ is an infinitesimal Maxwellian, i.e. of the form
$$
g(t,x,v)=\rho(t,x) +u(t,x)\cdot v+\th(t,x)\tfrac12(|v|^2-3)\,.
$$

Taking limits in the local conservation of mass leads then to
$$
\Div_x\la vg\ra =0\,,
$$
or in other words
$$
\Div_xu  =0
$$
which is the incompressibility constraint.

Multiplying the approximate momentum  equation (\ref{Mmnt-v}) by $\eps$
$$
\eps\d_t\la v_{K_\eps}g_\eps\g_\eps\ra+\eps \Div_x\bF_\eps(A)
+\tfrac13\grad_x\la\tfrac13|v|^2_{K_\eps}g_\eps\g_\eps\ra=\eps \bD_\eps(v)\,,
$$
using Propositions \ref{DF-PR} and \ref{ASMP-PR} to control $\bD_\eps(v)$ 
and the remainder term in $\bF_\eps(A)$
$$
\bF_\eps(A)-2\La A\left(\Pi\frac{\sqrt{G_\eps}-1}{\eps}\right)^2\Ra
+2\LA\hat A\frac1{\eps^2}\left(\sqrt{G'_\eps G'_{\eps 1}}
	-\sqrt{G_\eps G_{\eps 1}}\right)\RA\to 0\,,
$$
 and estimating $\bFc(A)$ and $\bFd(A)$ by the entropy and entropy production 
 bounds (\ref{Entr-estm2})-(\ref{Entr-prd1})
$$
\La A\left(\Pi\frac{\sqrt{G_\eps}-1}{\eps}\right)^2\Ra=O(1)\hbox{ in }L^\infty_t(L^1_x)\,,
$$
$$
\LA\hat A\frac1{\eps^2}\left(\sqrt{G'_\eps G'_{\eps 1}}
	-\sqrt{G_\eps G_{\eps 1}}\right)\RA=O(1)_{L^2_{t,x}}\,,
$$ 
 we also obtain
$$
\nabla_x \la |v|^2 g \ra=0
$$
or equivalently, since $\la|v|^2g\ra=3(\rho+\th)\in L^\infty(\bR_+;L^2(\bR^3))$,
$$
\rho+\theta =0\,,
$$
which is the Boussinesq relation.
One therefore has
(\ref{g=}).

\bigskip\bigskip
\noindent
\underbar{\it Step 2}~: start from (\ref{Nrml-Sqrt-Bltz}) in the proof of Lemma 
\ref{ADVC-LM}~:
$$
(\eps\d_t+v\cdot\grad_x)\frac{\sqrt{\eps^\a+G_\eps}-1}{\eps}
	=\frac1{\eps^2}\frac1{2\sqrt{\eps^\a+G_\eps}}\cQ(G_\eps,G_\eps)
		=Q^1_\eps+Q^2_\eps\,.
$$
Recall that
\begin{equation}
\lb{Q1to0}
Q_\eps^1 \to 0 \hbox{ in } L^1(Mdvdxdt)\,.
\end{equation}
Next observe that
$$
Q^2_\eps=\frac{\sqrt{G_\eps}}{\sqrt{\eps^\a+G_\eps}}\iint\sqrt{G_{\eps 1}}
\frac{\sqrt{G'_\eps G'_{\eps 1}}-\sqrt{G_\eps G_{\eps 1}}}{\eps^2}b(v-v_1,\om)d\om M_1dv_1\,;
$$
Proposition \ref{UINT-PR} implies that
$$
\sqrt{G_\eps}\to 1\hbox{ in }L^2_{loc}(dtdx;L^2((1+|v|)^\b Mdv))\hbox{ as }\eps\to 0\,;
$$
this and the second limit in (\ref{Cnvg-gtq}) imply that
$$
\begin{aligned}
\iint\sqrt{G_{\eps_n1}}
\frac{\sqrt{G'_{\eps_n}G'_{\eps_n1}}-\sqrt{G_{\eps_n}G_{\eps_n1}}}{\eps_n^2}
    b(v-v_1,\om)d\om M_1dv_1
\\
\to \iint\tilde q b(v-v_1,\om)d\om M_1dv_1
\end{aligned}
$$
in $\wL^1_{loc}(dtdx;L^1(Mdv))$ as $n\to+\infty$. Since on the other hand
$$
\frac{\sqrt{G_\eps}}{\sqrt{\eps^\a+G_\eps}}\to 1\hbox{ a.e. as }\eps\to 0\hbox{ with }
0\le\frac{\sqrt{G_\eps}}{\sqrt{\eps^\a+G_\eps}}\le 1\,,
$$
we conclude from the Product Limit Theorem that
\begin{equation}
\label{Cnvg-Q2}
Q^2_{\eps_n}\to\iint\tilde q b(v-v_1,\om)d\om M_1dv_1
\end{equation}
in $\wL^1_{loc}(dtdx;L^1(Mdv))$ as $n\to+\infty$.

By (\ref{sqr-sqrea}), (\ref{gga-sqrt}) and (\ref{Cnvg-gtq})
$$
\frac{\sqrt{\eps^\a_n+G_\eps}-1}{\eps_n}\to\tfrac12g
$$
in $\wL^1_{loc}(dtdx;L^1((1+|v|^2)Mdv))$ whenever $\a\in ]1,2[$. Using 
(\ref{Q1to0}), (\ref{Cnvg-Q2}) and the convergence above, and passing to the 
limit in (\ref{Nrml-Sqrt-Bltz}) as $\eps_n\to 0$ leads to 
$$
\iint\tilde qb(v-v_1,\om)d\om M_1dv_1=\tfrac12 v\cdot\grad_xg\,,
$$
which is precisely the first equality in (\ref{q-lim}).
Finally, replacing $g$ by its expression (\ref{g=}) in the formula above leads to the second
equality in  (\ref{q-lim}).
\end{proof}
 
Since $\hat A$ and $\hat B\in L^2(aM dv)$, the second limit in 
(\ref{Cnvg-gtq})  and identity (\ref{q-lim}) show that
\begin{equation}
\label{Lmt-Dffs}
\begin{aligned}
\bFdn(A)&=-2\La\hat A\frac1{\eps^2}\cQ(\sqrt{G_\eps},\sqrt{G_\eps})\Ra
\\
&\to-\la\hat A\otimes A\ra :\grad_xu=-\nu(\grad_xu+(\grad_xu)^T)
\\
\bFdn(B)&=-2\La\hat B\frac1{\eps^2}\cQ(\sqrt{G_\eps},\sqrt{G_\eps})\Ra
\\
&\to-\la\hat B\otimes B\ra :\grad_x\th=-\ka\grad_x\th
\end{aligned}
\end{equation}
in $\wL^2(dtdx)$ as $\eps\to 0$, because of the divergence-free condition in (\ref{g=}).

\subsection{Convergence of the convection terms}
The goal of this section is to establish that
$$
\begin{aligned}
\int\bFc(A):\grad_xUdx\to\int u\otimes u:\grad_xU dx\\
\hbox{ and }\Div_x\bFc(B)\to\tfrac52\Div_x(u\th)
\end{aligned}
$$
in the sense of distributions on $\bR_+^*$ and on $\bR_+^*\times\bR^3$ respectively.

First, we replace $\bFc(A)$ and $\bFc(B)$ by asymptotically equivalent expressions.

Indeed, because of (\ref{gga-sqrt})
$$
\la vg_\eps\g_\eps\ra-2\La v\frac{\sqrt{G_\eps}-1}{\eps}\Ra\to 0\hbox{ in }L^2_{loc}(dtdx)
$$
and
$$
\la(\tfrac13|v|^2-1)g_\eps\g_\eps\ra-
    2\La(\tfrac13|v|^2-1)\frac{\sqrt{G_\eps}-1}{\eps}\Ra\to 0\hbox{ in }L^2_{loc}(dtdx)\,.
$$
On the other hand, $g_\eps\g_\eps$ is bounded in $L^\infty_t(L^2(Mdvdx))$ while 
$v\indc_{|v|^2>K_\eps}\to 0$ and $(\tfrac13|v|^2-1)\indc_{|v|^2>K_\eps}\to 0$ in $L^2(Mdv)$;
therefore
$$
\la vg_\eps\g_\eps\ra-\la v_{K_\eps}g_\eps\g_\eps\ra\to 0
\hbox{ and }
\la(\tfrac13|v|^2-1)g_\eps\g_\eps\ra-\la(\tfrac13|v|^2-1)_{K_\eps}g_\eps\g_\eps\ra\to 0
$$
in $L^2_{loc}(dtdx)$. Therefore
\begin{equation}
\label{FcA->}
\begin{aligned}
\bFc(A)
&=\tfrac12\la A\otimes A\ra : \la v_{K_\eps}g_\eps\g_\eps\ra^{\otimes 2}+o(1)_{L^1_{loc}(dtdx)}
\\
&=\la v_{K_\eps}g_\eps\g_\eps\ra^{\otimes 2}
	-\tfrac13|\la v_{K_\eps}g_\eps\g_\eps\ra|^2I
+o(1)_{L^1_{loc}(dtdx)}\,,
\end{aligned}
\end{equation}
while
\begin{equation}
\label{FcB->}
\begin{aligned}
\bFc(B)
&=\la B\otimes B\ra\cdot\la v_{K_\eps}g_\eps\g_\eps\ra
    \la(\tfrac13|v|^2-1)_{K_\eps}g_\eps\g_\eps\ra+o(1)_{L^1_{loc}(dtdx)}
\\
&=\tfrac52\la(\tfrac13|v|^2-1)_{K_\eps}g_\eps\g_\eps\ra\la v_{K_\eps}g_\eps\g_\eps\ra
+o(1)_{L^1_{loc}(dtdx)}\,.
\end{aligned}
\end{equation}

Furthermore, since $g_{\eps_n}\to g$ weakly in $L^1_{loc}(dtdx;L^1((1+|v|^2)Mdv))$ while
$$
\begin{aligned}
v_{K_\eps}\g_\eps\to v\hbox{ and }
(\tfrac13|v|^2-1)_{K_{\eps}}\g_\eps\to(\tfrac13|v|^2-1)\hbox{ a.e., and}
\\
|v_{K_\eps}\g_\eps|+|(\tfrac13|v|^2-1)_{K_{\eps}}\g_\eps|\le C(1+|v|^2)
\end{aligned}
$$
one has by the Product Limit Theorem
\begin{equation}
\label{Lmt-uth}
\begin{aligned}
\la v_{K_{\eps_n}}\g_{\eps_n} g_{\eps_n}\ra&\to\la vg\ra=u
\\
\la (\tfrac13|v|^2-1)_{K_{\eps_n}}\g_{\eps_n} g_{\eps_n}\ra
	&\to\la (\tfrac13|v|^2-1)g\ra=\th
\end{aligned}
\end{equation}
in $\wL^1_{loc}(dtdx)$. In fact, these limits also hold in $\wL^2_{loc}(dtdx)$ since the family
$g_\eps\g_\eps$ is bounded in $L^\infty_t(L^2(Mdvdx))$.

\bigskip
Taking limits in (\ref{FcA->}) and (\ref{FcB->}), which are quadratic functions of the moments,  
requires then to establish some strong compactness on $(\la \zeta_{K_\eps} g_\eps \g_\eps \ra)$.

\subsubsection{Strong compactness in the $x$-variable}

Applying Proposition \ref{CPTX-PR} with $\xi=v$ and $\xi=\tfrac12(|v|^2-5)$ shows that, 
for each $T>0$ and each compact $K\subset \bR^3$,
$$
\begin{aligned}
\int_0^T\int_K|\la\tfrac12(|v|^2-5)g_{\eps_n}\g_{\eps_n}\ra(t,x+y)
    -\la\tfrac12(|v|^2-5)g_{\eps_n}\g_{\eps_n}\ra(t,x)|^2dxdt
\\
+\int_0^T\int_K|\la vg_{\eps_n}\g_{\eps_n}\ra(t,x+y)
    -\la vg_{\eps_n}\g_{\eps_n}\ra(t,x)|^2dxdt
    \to 0
\end{aligned}
$$
as $|y|\to 0$ uniformly in $n$. An easy consequence of the above convergence properties
is that
\begin{equation}
\label{cptx2}
\begin{aligned}
\int_0^T\int_K|\la\tfrac12(|v|^2-5)_{K_{\eps_n}}g_{\eps_n}\g_{\eps_n}\ra(t,x+y)
    -\la\tfrac12(|v|^2-5)_{K_{\eps_n}}g_{\eps_n}\g_{\eps_n}\ra(t,x)|^2dxdt
\\
+\int_0^T\int_K|\la v_{K_{\eps_n}}g_{\eps_n}\g_{\eps_n}\ra(t,x+y)
    -\la v_{K_{\eps_n}}g_{\eps_n}\g_{\eps_n}\ra(t,x)|^2dxdt
    \to 0
\end{aligned}
\end{equation}
as $|y|\to 0$ uniformly in $n$.

\bigskip
In order to study the convergence of $\bF_\eps(A)$, we need some similar statements 
for the solenoidal and gradient parts of $\la v_{K_{\eps_n}}g_{\eps_n}\g_{\eps_n}\ra$, 
since the first one is expected to converge strongly in $L^2_{loc}(dtdx)$.

The difficulty comes then from the fact that the Leray projection is a non local 
pseudodifferential operator, in particular it is not continuous on $L^2_{loc}(dx)$.

Introducing some convenient truncation $\chi$ in $x$ and using the properties of 
the commutator $[\chi,P]$, one can nevertheless prove the following equicontinuity
statement (see Lemma \ref{P-equicontinuity})~: for each compact  $K\subset\bR^3$ 
and each $T>0$, one has 
\begin{equation}
\label{cptx}
\int_0^T\int_K|P\la v_{K_{\eps_n}}g_{\eps_n}\g_{\eps_n}\ra(t,x+y)
    -P\la v_{K_{\eps_n}}g_{\eps_n}\g_{\eps_n}\ra(t,x)|^2dxdt
    \to 0
\end{equation}
as $|y|\to 0$, uniformly in $n$.

\subsubsection{Strong compactness  in the $t$-variable} 

As we shall see below , the temperature fluctuation 
$\la\tfrac12(|v|^2-5)_{K_\eps}g_{\eps_n}\g_{\eps_n}\ra$  and the solenoidal 
part $P\la v_{K_\eps}g_\eps\g_\eps\ra$ of $\la v_{K_\eps}g_\eps\g_\eps\ra$ 
are strongly compact in the $t$-variable. However the orthogonal complement 
of  $P\la v_{K_\eps}g_\eps\g_\eps\ra$--- which is a gradient field --- is not in 
general because of high frequency oscillations in $t$.

\begin{Prop}\label{CPTPU-PR}
Under the assumptions of Theorem \ref{BNSW-TH}, one has
$$
\begin{aligned}
P\la v_{K_{\eps_n}} g_{\eps_n}\g_{\eps_n}\ra&\to\la vg\ra=u
\\
\la\tfrac12(|v|^2-5)_{K_{\eps_n}} g_{\eps_n}\g_{\eps_n}\ra
	&\to\la \tfrac12(|v|^2-5) g\ra=\tfrac52\th
\end{aligned}
$$
in $C(\bR_+;\wL^2_x)$ and in $L^2_{loc}(dtdx)$ as $n\to+\infty$.
\end{Prop}

\begin{proof}

The conservation law (\ref{Mmnt-Pv->0}) implies that
\begin{equation}
\label{CptTPu}
\d_t\int_{\bR^3}\la v_{K_{\eps_n}}g_{\eps_n}\g_{\eps_n}\ra\cdot U dx=O(1)
\hbox{ in }L^1_{loc}(dt)
\end{equation}
for each compactly supported, solenoidal vector field $U\in C^\infty(\bR^3)$, since we know 
from Proposition \ref{ASMP-PR} together with the bounds (\ref{Entr-prd1}) and (\ref{Entr-estm2}) 
that $\bF_{\eps_n}(A)$ is bounded in $L^1_{loc}(dtdx)$.  

In the same way, the conservation law (\ref{Mmnt-v2->0}) implies that
\begin{equation}
\label{CptTth}
\d_t\la\tfrac12(|v|^2-5)_{K_{\eps_n}}g_{\eps_n}\g_{\eps_n}\ra=O(1)
\hbox{ in }L^1_{loc}(dt;W^{-1,1}_{loc}(\bR^3))\,.
\end{equation}

Also, we recall that $g_\eps\g_\eps$ is bounded in $B(\bR_+;L^2(Mdvdx))$ --- 
where $B(X,Y)$ denotes the class of bounded maps from $X$ to $Y$ --- because 
of the entropy bound (\ref{Entr-estm2}). Indeed, since $\g_\eps=0$ whenever $G_\eps>2$, 
one has
\begin{equation}
\label{gflt<sqrG}
|g_{\eps_n}\g_{\eps_n}|
\le\indc_{G_\eps\le 2}\frac{|G_\eps-1|}{\eps}\le(1+\sqrt2)\frac{|\sqrt{G_{\eps_n}}-1|}{\eps_n}\,.
\end{equation}
In particular, one has
\begin{equation}
\label{Bdgflt}
\begin{aligned}
\la v_{K_{\eps_n}}g_{\eps_n}\g_{\eps_n}\ra=O(1)\hbox{ in }B(\bR_+;L^2_x)\,,\\
\la\tfrac12(|v|^2-5)_{K_{\eps_n}}g_{\eps_n}\g_{\eps_n}\ra=O(1)
\hbox{ in }B(\bR_+;L^2_x)\,,
\end{aligned}
\end{equation}

\bigskip
Since the class of $C^\infty$, compactly supported solenoidal vector fields is dense in that of all
$L^2$ solenoidal vector fields (see Appendix A of \cite{Li2}), (\ref{Bdgflt}) and
(\ref{CptTPu}) imply that
\begin{equation}
\label{CptCwL2}
P\la v_{K_{\eps_n}}g_{\eps_n}\g_{\eps_n}\ra
\hbox{ is relatively compact in }C(\bR_+;\wL^2(\bR^3))\,,
\end{equation}
by a variant of Ascoli's theorem that can be found in Appendix C of \cite{Li2}.

The same argument shows that
\begin{equation}
\label{CptCwL22}
\la\tfrac12(|v|^2-5)_{K_{\eps_n}}g_{\eps_n}\g_{\eps_n}\ra 
	\hbox{  is also relatively compact in }C(\bR_+;\wL^2_x)\,.
\end{equation}

\bigskip
As for the $L^2_{loc}(dtdx)$ compactness, notice (\ref{CptCwL2})-(\ref{CptCwL22}) 
imply that
$$
\begin{aligned}
P\la v_{K_{\eps_n}}g_{\eps_n}\g_{\eps_n}\ra\star\chi_\de
	\hbox{ is relatively compact in }L^2_{loc}(dtdx)
\\
\la\tfrac12(|v|^2-5)_{K_{\eps_n}}g_{\eps_n}\g_{\eps_n}\ra\star \chi_\de
	\hbox{ is relatively compact in }L^2_{loc}(dtdx)\\
\end{aligned}
$$
where $\chi_\de$ designates any mollifying sequence and $\star$ is the convolution 
in the $x$-variable only. Hence
$$
\begin{aligned}
P\la v_{K_{\eps_n}}g_{\eps_n}\g_{\eps_n}\ra
\cdot P\la v_{K_{\eps_n}}g_{\eps_n}\g_{\eps_n}\ra\star\chi_\de
	&\to Pu\cdot Pu\star\chi_\de
\\
\la\tfrac12(|v|^2-5)_{K_{\eps_n}}g_{\eps_n}\g_{\eps_n}\ra 
	\la\tfrac12(|v|^2-5)_{K_{\eps_n}}g_{\eps_n}\g_{\eps_n}\ra\star \chi_\de 
		&\to\left(\tfrac52 \theta\right)\left(\tfrac52\theta\star \chi_\de\right)
\end{aligned}
$$
in $\wL^1_{loc}(dtdx)$ as $n\to \infty$. By (\ref{cptx2})-(\ref{cptx}),
\begin{equation}
\label{UnifConvxPu-th}
\begin{aligned}
P\la v_{K_{\eps_n}}g_{\eps_n}\g_{\eps_n}\ra\star\chi_\de
	&\to P\la v_{K_{\eps_n}}g_{\eps_n}\g_{\eps_n}\ra
\\
\la\tfrac12(|v|^2-5)_{K_{\eps_n}}g_{\eps_n}\g_{\eps_n}\ra\star\chi_\de 
&\to\la\tfrac12(|v|^2-5)_{K_{\eps_n}}g_{\eps_n}\g_{\eps_n}\ra
\end{aligned}
\end{equation}
in $L^2_{loc}(dtdx)$ uniformly in $n$ as $\de\to 0$. With this, we conclude that
$$
\begin{aligned}
|P\la v_{K_{\eps_n}}g_{\eps_n}\g_{\eps_n}\ra|^2
	&\to|Pu|^2\hbox{ in }\wL^1_{loc}(dtdx)
\\
|\la\tfrac12(|v|^2-5)_{K_{\eps_n}}g_{\eps_n}\g_{\eps_n}\ra|^2
	&\to \left( \frac52 \theta\right)^2\hbox{ in }\wL^1_{loc}(dtdx)
\end{aligned}
$$
which implies the expected strong compactness in $L^2_{loc}(dtdx)$.
\end{proof}

\subsubsection{Passing to the limit in the convection terms}

As explained above, $P\la v_{K_{\eps_n}}g_{\eps_n}\g_{\eps_n}\ra$ is strongly 
relatively compact in $L^2_{loc}(dtdx)$; however, the term 
$\la v_{K_{\eps_n}}g_{\eps_n}\g_{\eps_n}\ra$ itself may not be strongly relatively 
compact in $L^2_{loc}(dtdx)$ --- at least in general. For that reason, on account 
of (\ref{FcA->}), it is not clear that 
$$
\bFc(A)\to u\otimes u-\tfrac13|u|^2I\,.
$$ 

Likewise $\la(|v|^2-5)_{K_{\eps_n}}g_{\eps_n}\g_{\eps_n}\ra$ is strongly relatively 
compact in $L^2_{loc}(dtdx)$, and, on account of (\ref{FcB->}), it is not clear that
$$
\bFc(B)\to \tfrac52 u\th
$$ 
as one would expect. 

What we shall prove in this section is 

\begin{Prop}\label{DVCNV-PR}
Under the assumptions of Theorem \ref{BNSW-TH}, one has
$$
\int_{\bR^3}\grad_xU:\bFcn(A)dx\to\int_{\bR^3}\grad_xU:u\otimes udx
$$
in the sense of distributions on $\bR_+^*$ for each solenoidal vector field
$U\in C^\infty_c(\bR^3;\bR^3)$,  and
$$
\Div_x\bFcn(B)\to\tfrac52\Div_x(u\th)
$$
in the sense of distributions on $\bR_+^*\times\bR^3$.
\end{Prop}

The proof of this result relies on a compensated compactness argument due to 
P.-L. Lions and N. Masmoudi \cite{LiMa1} and recalled in Appendix A (Theorem A.2), 
and on the following observation:

\begin{Lem}\label{CMPCPT-LM}
Let $\de>0$, and $\xi\in C^\infty_c(\bR^3)$ be a bump function such that
$$
\Supp\xi\subset B(0,1)\,,\quad\xi\ge 0\,,\hbox{ and }\int\xi dx=1\,;
$$
let $\xi_\de(x)=\de^{-3}\chi(x/\de)$ and $\l_\de=\xi_\de\star\xi_\de\star\xi_\de$. Denote 
by $Q=I-P$  the orthogonal projection on gradient fields in $L^2(\bR^3;\bR^3)$. Under 
the assumptions of Theorem \ref{BNSW-TH}, one has
$$
\begin{aligned}
\eps\d_tQ(\l_\de\star\la v_{K_{\eps_n}}g_{\eps_n}\g_{\eps_n}\ra)
+
\grad_x\l_\de\star\la\tfrac13|v|^2_{K_{\eps_n}}g_{\eps_n}\g_{\eps_n}\ra&\to 0
\\
\eps\d_t\l_\de\star\la\tfrac13|v|^2_{K_{\eps_n}}g_{\eps_n}\g_{\eps_n}\ra
+
\tfrac53\Div_xQ(\l_\de\star\la v_{K_{\eps_n}}g_{\eps_n}\g_{\eps_n}\ra)&\to 0
\end{aligned}
$$
in $L^1_{loc}(dt;H^s_{loc}(\bR^3))$ for each $s>0$.
\end{Lem} 

\begin{proof}
The second convergence statement above is obvious: indeed, considering the truncated,
renormalized energy equation (\ref{Mmnt-xi}) with $\xi=\tfrac13|v|^2$, and applying the 
mollifier $\l_\de$ leads to
$$
\begin{aligned}
\eps\d_t\l_\de\star\la\tfrac13|v|^2_{K_{\eps_n}}g_{\eps_n}\g_{\eps_n}\ra
+
\tfrac53\Div_xQ(\l_\de\star\la v_{K_{\eps_n}}g_{\eps_n}\g_{\eps_n}\ra)&
\\
=-\tfrac23\eps\Div_x\l_\de\star\bF_\eps(B)+\tfrac13\eps\l_\de\star\bD_\eps(|v|^2)&\,.
\end{aligned}
$$
It follows from Proposition \ref{ASMP-PR}, the entropy bound (\ref{Entr-estm2}) 
and the entropy production estimate (\ref{Entr-prd1}) that $\bF_\eps(B)$ is bounded 
in $L^1_{loc}(dtdx)$; this and Proposition \ref{DF-PR} eventually entail that the 
right-hand side of the above equality vanishes in $L^1_{loc}(dt;H^s_{loc}(\bR^3))$.

\bigskip
The first convergence statement above is much trickier. Start from the analogous truncated, 
renormalized momentum equation (\ref{Mmnt-xi}) with $\xi=v$:
\begin{equation}
\label{epsMmnt-v}
\begin{aligned}
\eps\d_t\la v_{K_{\eps_n}}g_{\eps_n}\g_{\eps_n}\ra
+
\grad_x\tfrac1{3\eps}\la|v|^2_{K_{\eps_n}}g_{\eps_n}\g_{\eps_n}\ra
=-\eps\Div_x\bF_\eps(A)+\eps\bD_\eps(v)&
\end{aligned}
\end{equation}
Applying $Q$ to both sides of the equality above is not obvious, because we only 
know that the right hand side vanishes in $L^1_{loc}(dt;W^{-1,1}_{loc}(\bR^3))$, 
while $Q$ is known to be continuous on { global} Sobolev spaces only.

However, $Q=\grad_x\Dlt_x^{-1}\Div_x$ is a singular integral operator whose 
integral kernel decays at infinity. More precisely, we shall use Lemma \ref{LM-Q} 
together with the following estimates on the right hand side of (\ref{epsMmnt-v})~:

\begin{Lem}\label{LM-RHS}
One has
$$
\eps \bF_\eps (A) \to 0 \hbox{ and }\eps \bD_\eps(v)\to 0\hbox{ in }L^1_{loc}(dtdx)
$$
as $\eps\to 0$. Furthermore,
$$
\begin{aligned}
\eps \bF_\eps(A)=O(1)_{L^\infty_t(L^2_x)}\,,
\\
\eps \bD_\eps(v) =O(\eps^2K_\eps^{1/2})_{L^1_{t,x}}
	+O(\sqrt{\eps})_{L^2_t(L^1_x)}+O(1)_{L^2_{t,x}}\,.
\end{aligned}
$$
\end{Lem}

Note that the convergence statement in Lemma \ref{LM-RHS} is a simple 
consequence of Propositions  \ref{DF-PR} and \ref{ASMP-PR} (already 
used in the derivation of the Boussinesq relation in paragraph \ref{diffusion-cv}).

Then let us postpone the proof of the global estimates, which is based on 
the entropy and entropy production bounds (\ref{Entr-estm2})-(\ref{Entr-prd1}), 
and conclude the proof of Lemma \ref{CMPCPT-LM}.

Define $\zeta_\de=\xi_\de\star\xi_\de$. One has then
$$
\begin{aligned}
\eps\d_tQ(\zeta_\de\star\la v_{K_\eps}g_\eps\g_\eps\ra)
+
\grad_x\zeta_\de\star\la\tfrac13|v|^2_{K_\eps}g_\eps\g_\eps\ra
=
-Q(\xi_\de\star(\eps \bF_\eps(A)\star\grad\xi_\de))
\\
+
Q(\zeta_\de\star(\eps \bD_\eps(v)))\,.
\end{aligned}
$$
For each $\de>0$ fixed,
$$
Q(\xi_\de\star(\eps\bF_\eps(A)\star\grad\xi_\de))\to 0
	\hbox{ in }L^1_{loc}(dtdx)\hbox{ as }\eps\to 0
$$
by the first convergence result in Lemma \ref{LM-RHS} and Lemma \ref{LM-Q}. 

Next decompose 
$$
\eps \bD_\eps(v)=\bD_\eps^0(v)+\bD_\eps'(v)
$$ 
with 
$$
\bD_\eps^0(v)=O(1)_{L^2_{t,x}}\hbox{ and }
\bD_\eps'(v)=O(\eps^2K_\eps^{1/2})_{L^1_{t,x}}+O(\sqrt{\eps})_{L^2_t(L^1_x)}\,.
$$
Thus, one has
$$
\zeta_\de\star \bD_\eps'(v)\to 0\hbox{ in }L^1_{loc}(dt;L^2_x)
$$
so that
$$
Q(\zeta_\de\star\bD_\eps'(v))\to 0\hbox{ in }L^1_{loc}(dt;L^2_x)
$$
as $\eps\to 0$, by the $L^2$-continuity of pseudo-differential operators of order 0. 
Finally, since $\bD_\eps^0(v)\to 0$ in $L^1_{loc}(dtdx)$ and is bounded in $L^2_{t,x}$, 
it follows from Lemma \ref{LM-Q} that
$$
Q(\zeta_\de\star\bD_\eps^0(v))\to 0\hbox{ in }L^1_{loc}(dtdx)\,.
$$

Eventually, we have proved that
$$
\eps\d_tQ(\zeta_\de\star\la v_{K_\eps}g_\eps\g_\eps\ra)
+
\grad_x\zeta_\de\star\la\tfrac13|v|^2_{K_\eps}g_\eps\g_\eps\ra
\to 0
$$
in $L^1_{loc}(dtdx)$ as $\eps\to 0$. Therefore, denoting $\l_\de=\xi_\de\star\xi_\de\star\xi_\de$,
one has
$$
\eps\d_tQ(\l_\de\star\la v_{K_\eps}g_\eps\g_\eps\ra)
+
\grad_x\l_\de\star\la\tfrac13|v|^2_{K_\eps}g_\eps\g_\eps\ra
\to 0
$$
in $L^1_{loc}(dt;H^s_{loc}(\bR^3))$ for each $s>0$ as $\eps\to 0$.
\end{proof}

\bigskip
Let us now turn to the 
\begin{proof}[Proof of Lemma \ref{LM-RHS}]
First, $g_\eps\g_\eps=O(1)$ in $L^\infty_t(L^2(dxMdv))$, while $A\in L^2(Mdv)$: 
hence 
$$
\la A_{K_\eps}g_\eps\g_\eps\ra=O(1)_{L^\infty_t(L^2_x)}\,.
$$

Next decompose $\eps\bD_\eps(v)$ as
$$
\eps\bD_\eps(v)=T_1+T_2+T_3
$$
where
$$
\begin{aligned}
T_1=\LA v_{K_\eps}\hat\g_\eps\frac1{\eps^2}
\left(\sqrt{G'_\eps G'_{\eps 1}}-\sqrt{G_\eps G_{\eps 1}}\right)^2\RA\,,
\\
T_2=2\LA v_{K_\eps}\hat\g_\eps\sqrt{G_\eps}
    \frac1{\eps^2}\left(\sqrt{G'_\eps G'_{\eps 1}}-\sqrt{G_\eps G_{\eps 1}}\right)\RA\,,
\\
T_3=
2\LA v_{K_\eps}\hat\g_\eps\sqrt{G_\eps}\left(\sqrt{G_{\eps 1}}-1\right)
    \frac1{\eps^2}\left(\sqrt{G'_\eps G'_{\eps 1}}-\sqrt{G_\eps G_{\eps 1}}\right)\RA\,,
\end{aligned}
$$

Since $\frac1{\eps^4}\left(\sqrt{G'_\eps G'_{\eps 1}}-\sqrt{G_\eps G_{\eps 1}}\right)^2$ 
is bounded in $L^1_{t,x,\mu}$ (see (\ref{Entr-prd1})) , one has 
$$
T_1=O(\eps^2K_\eps^{1/2})_{L^1_{t,x}}\,.
$$

Likewise, $\hat\g_\eps\sqrt{G_\eps}=O(1)$ in $L^\infty_{t,x,v}$ and $v\in L^2(d\mu)$, 
so that 
$$
T_2=O(1)_{L^2_{t,x}}\,.
$$

The same argument is used for $T_3$, except that one has to control the terms 
$v\left(\sqrt{G_{\eps 1}}-1\right)$ instead of $v$ in $L^2_\mu$. By Young's inequality,
$$
\begin{aligned}
(1+|v_1|)\left(\sqrt{G_{\eps 1 }}-1\right)^2
		&\le (1+|v_1|)\left|G_{\eps 1}-1\right| 
\\
&\leq \frac1\eps\left(h(G_{\eps 1}-1)+h^*\left(\eps (1+|v_1|)\right)\right)
\\
&
\le \frac1\eps h(\eps g_{\eps 1})+\eps h^*(1+|v_1|) 
\\
&=O( \eps )_{L^\infty_t (L^1(M_1dv_1dx))} 
	+ O(\eps )_{L^\infty_{t,x}(L^1(M_1 dv_1))}
\end{aligned}
$$
The 3rd inequality above comes from the superquadratic nature of $h^*$. Indeed
$$
h^*(p)=e^{p}-p-1=\sum_{n\ge 2}\frac{p^n}{n!}
$$
so that
$$
h^*(\l p)\le\l^2h^*(p)\,,\quad\hbox{ for each }p\ge 0\hbox{ and }\l\in[0,1]\,.
$$

With the upper bound on $\int b(v-v_1,\om)d\om$, this shows that
$$
\begin{aligned}
|T_3|
&\leq \| v_{K_\eps}\|_{L^2((1+|v|)^\b M dv)} \| \hat \g_\eps G_\eps \|_{L^\infty_v} 
\|\sqrt{G_{\eps 1}}-1\|_{L^2((1+|v_1|)^\b M_1 dv_1)} 
\\
&\qquad \left\|  \frac1{\eps^2}\left(\sqrt{G'_\eps G'_{\eps 1}}
	-\sqrt{G_\eps G_{\eps 1}}\right)\right\|_{L^2_\mu}
\\
&= O(\sqrt{\eps})_ {L^2_t(L^1_x)} +O(\sqrt{\eps})_ {L^2_{t,x}}\,.
\end{aligned}
$$
Combining the previous results leads to  the expected estimate for $ \bD_\eps(v)$.
\end{proof}

\bigskip\bigskip
At this point, we conclude this section with the 

\begin{proof}[Proof of Proposition \ref{DVCNV-PR}]
First, we apply the compensated compactness argument for the acoustic system in
\cite{LiMa1} --- see also Theorem \ref{A2-thm} --- to conclude from the statement in
Lemma \ref{CMPCPT-LM} that
$$
\begin{aligned}
\int\grad_xU:Q(\l_\de\star\la v_{K_{\eps_n}}g_{\eps_n}\g_{\eps_n}\ra)^{\otimes 2}dx
&\to 0
\\
\Div_x(\l_\de\star\la\tfrac13|v|^2_{K_{\eps_n}}g_{\eps_n}\g_{\eps_n}\ra
	Q(\l_\de\star\la v_{K_{\eps_n}}g_{\eps_n}\g_{\eps_n}\ra))&\to 0
\end{aligned}
$$
in the sense of distributions on $\bR_+^*$ and $\bR_+^*\times\bR^3$ respectively,
for each divergence-free vector field $U\in C^\infty_c(\bR^3;\bR^3)$. 

On the other hand, the compactness property in the $x$-variable stated in Proposition
\ref{CPTX-PR} and (\ref{UnifConvxPu-th}) implies that
$$
\begin{aligned}
Q(\l_\de\star\la v_{K_{\eps_n}}g_{\eps_n}\g_{\eps_n}\ra
-Q(\la v_{K_{\eps_n}}g_{\eps_n}\g_{\eps_n}\ra\to 0
\\
\l_\de\star\la\tfrac13|v|^2_{K_{\eps_n}}g_{\eps_n}\g_{\eps_n}\ra
-\la\tfrac13|v|^2_{K_{\eps_n}}g_{\eps_n}\g_{\eps_n}\ra\to 0
\end{aligned}
$$
in $L^2_{loc}(dtdx)$ as $\de\to 0$, uniformly in $n$. Therefore, one has
\begin{equation}
\label{osc2}
\begin{aligned}
\int\grad_xU:Q(\la v_{K_{\eps_n}}g_{\eps_n}\g_{\eps_n}\ra)^{\otimes 2}dx
&\to 0
\\
\Div_x(\la\tfrac13|v|^2_{K_{\eps_n}}g_{\eps_n}\g_{\eps_n}\ra
Q(\la v_{K_{\eps_n}}g_{\eps_n}\g_{\eps_n}\ra))&\to 0
\end{aligned}
\end{equation}
in the sense of distributions on $\bR_+^*$ and $\bR_+^*\times\bR^3$ respectively,
for each divergence-free vector field $U\in C^\infty_c(\bR^3;\bR^3)$.

Also, we recall from Proposition \ref{CPTPU-PR} and (\ref{Lmt-uth}) that
$$
\begin{aligned}
P\la v_{K_{\eps_n}}g_{\eps_n}\g_{\eps_n}\ra\to u
\hbox{ strongly in }L^2_{loc}(dtdx)\,,
\\
Q\la v_{K_{\eps_n}}g_{\eps_n}\g_{\eps_n}\ra\to 0
\hbox{ weakly in }L^2_{loc}(dtdx)\,.
\end{aligned}
$$
Therefore, for each compactly supported, $C^\infty$ solenoidal vector field $U$, one has
$$
\begin{aligned}
\int_{\bR^3}\grad_xU:\la v_{K_{\eps_n}}g_{\eps_n}\g_{\eps_n}\ra^{\otimes 2}dx
=
\int_{\bR^3}\grad_xU:(P\la v_{K_{\eps_n}}g_{\eps_n}\g_{\eps_n}\ra)^{\otimes 2}dx
\\
+
\int_{\bR^3}\grad_xU:(Q\la v_{K_{\eps_n}}g_{\eps_n}\g_{\eps_n}\ra)^{\otimes 2}dx
\\
+
\int_{\bR^3}\grad_xU:(P\la v_{K_{\eps_n}}g_{\eps_n}\g_{\eps_n}\ra\otimes
Q\la v_{K_{\eps_n}}g_{\eps_n}\g_{\eps_n}\ra)dx
\\
+
\int_{\bR^3}\grad_xU:(Q\la v_{K_{\eps_n}}g_{\eps_n}\g_{\eps_n}\ra\otimes
P\la v_{K_{\eps_n}}g_{\eps_n}\g_{\eps_n}\ra)dx
\\
\to\int_{\bR^3}\grad_xU:u\otimes u dx
\end{aligned}
$$
in the sense of distributions on $\bR_+^*$. Together with (\ref{FcA->}), this implies the first
limit in Proposition \ref{DVCNV-PR}.

On the other hand, Proposition \ref{CPTPU-PR} and (\ref{Lmt-uth}) imply that
$$
\begin{aligned}
\la(\tfrac15|v|^2-1)_{K_{\eps_n}}g_{\eps_n}\g_{\eps_n}\ra\to\th
\hbox{ strongly in }L^2_{loc}(dtdx)\,,
\\
\la|v|^2_{K_{\eps_n}}g_{\eps_n}\g_{\eps_n}\ra\to 0
\hbox{ weakly in }L^2_{loc}(dtdx)\,.
\end{aligned}
$$
Hence
$$
\begin{aligned}
\Div_x(\la(\tfrac13|v|^2-1)_{K_{\eps_n}}g_{\eps_n}\g_{\eps_n}\ra
\la v_{K_{\eps_n}}g_{\eps_n}\g_{\eps_n}\ra)
\\
=
\Div_x(\la(\tfrac15|v|^2-1)_{K_{\eps_n}}g_{\eps_n}\g_{\eps_n}\ra
P\la v_{K_{\eps_n}}g_{\eps_n}\g_{\eps_n}\ra)
\\
+
\tfrac{2}{15}\Div_x(\la|v|^2_{K_{\eps_n}}g_{\eps_n}\g_{\eps_n}\ra
Q\la v_{K_{\eps_n}}g_{\eps_n}\g_{\eps_n}\ra)
\\
+
\tfrac{2}{15}\Div_x(\la|v|^2_{K_{\eps_n}}g_{\eps_n}\g_{\eps_n}\ra
P\la v_{K_{\eps_n}}g_{\eps_n}\g_{\eps_n}\ra)
\\
+
\Div_x(\la(\tfrac15|v|^2-1)_{K_{\eps_n}}g_{\eps_n}\g_{\eps_n}\ra
Q\la v_{K_{\eps_n}}g_{\eps_n}\g_{\eps_n}\ra)
\\
\to \Div_x(u\th)
\end{aligned}
$$
in the sense of distributions on $\bR_+^*\times\bR^3$. With (\ref{FcB->}), this
entails the second statement in Proposition \ref{DVCNV-PR}.
\end{proof}

\subsection{End of the proof of Theorem \ref{BNSW-TH}}

At this point we return to the renormalized, truncated momentum and energy conservations
in the form (\ref{Mmnt-Pv->0}) and (\ref{Mmnt-v2->0}).

\bigskip
\noindent
\underbar{ Asymptotic conservation of momentum}~: 
by using the convergence properties in (\ref{Lmt-Dffs}) and Proposition \ref{DVCNV-PR} with
the decomposition (\ref{FlxDcmp}), one sees that, for each $C^\infty$, compactly supported, 
solenoidal vector field $U$,
$$
\int_{\bR^3}\grad_xU:\bF_{\eps_n}(A)dx\to\int_{\bR^3}\grad_xU:u\otimes udx
-\nu\int_{\bR^3}\grad_xU:(\grad_xu+(\grad_xu)^T)dx
$$
in the sense of distributions on $\bR_+^*$, while
$$
\Div_x\bF_{\eps_n}(B)\to\Div_x(u\th)-\ka\Dlt_x\th
$$
in the sense of distributions in $\bR_+^*\times\bR^3$. Furthermore, since $\Div_xu=0$, one
has
$$
\int_{\bR^3}\grad_xU:(\grad_xu)^Tdx=\int_{\bR^3}\grad_x(\Div_xU)\cdot udx=0
$$
for each solenoidal test vector field $U$, so that
$$
\int_{\bR^3}\grad_xU:\bF_{\eps_n}(A)dx\to\int_{\bR^3}\grad_xU:u\otimes udx
-\nu\int_{\bR^3}\grad_xU:\grad_xudx
$$
in the sense of distributions on $\bR_+^*$.

On the other hand, by Proposition \ref{CPTPU-PR},
$$
\int_{\bR^3}U\cdot\la v_{K_{\eps_n}}g_{\eps_n}\g_{\eps_n}\ra dx\to\int_{\bR^3}U\cdot udx
$$
uniformly on $[0,T]$ for each $T>0$. In particular, for $t=0$, one has 
$$
\int_{\bR^3}U\cdot u\rstr_{t=0}dx
=\lim_{\eps\to 0}\int_{\bR^3}U\cdot P\left(\frac1\eps\int_{\bR^3}vF^{in}_\eps dv\right)dx
=\int_{\bR^3}U\cdot u^{in}dx\,.
$$
Therefore, $u$ satisfies
$$
\begin{aligned}
\d_t\int_{\bR^3}U\cdot udx
-\int_{\bR^3}\grad_xU:u\otimes udx+\nu\int_{\bR^3}\grad_xU:\grad_xudx&=0\,,\quad t>0\,,
\\
u\rstr_{t=0}&=u^{in}\,.
\end{aligned}
$$

\bigskip
\noindent
\underbar{ Asymptotic conservation of energy}~: 
likewise,
$$
\la(\tfrac15|v|^2-1)_{K_{\eps_n}}g_{\eps_n}\g_{\eps_n}\ra\to\th
$$
in $C(\bR_+;w-L^2_x)$. In particular, for $t=0$, one has
$$
\th\rstr_{t=0}=w-\lim_{\eps\to 0}\frac1\eps\int_{\bR^3}(\tfrac15|v|^2-1)F^{in}_\eps dv=\th^{in}\,.
$$
Therefore, $\th$ satisfies
$$
\begin{aligned}
\d_t\th+\Div_x(u\th)-\ka\Dlt_x\th&=0\,,\quad x\in\bR^3\,,\,\,t>0\,,
\\
\th\rstr_{t=0}&=\th^{in}
\end{aligned}
$$

Notice that one has also
$$
\begin{aligned}
\frac1{\eps_n}\int_{\bR^3}vF_{\eps_n}dv=\la vg_{\eps_n}\ra\to u
\\
\frac1{\eps_n}\int_{\bR^3}(\tfrac15|v|^2-1)(F_{\eps_n}-M)dv
	=\la(\tfrac15|v|^2-1)g_{\eps_n}\ra\to\th
\end{aligned}
$$
weakly in $L^1_{loc}(dtdx)$, because of (\ref{Cnvg-gtq}) and (\ref{g=}).

\bigskip
\noindent
\underbar{ Asymptotic energy inequality}~: by Proposition \ref{TQ=-LM} and (\ref{identity1}), 
one has
$$ 
\frac{2}{\eps_n} (\sqrt{G_{\eps_n}}-1) \to g \hbox{ in }\wL_{loc}^2(dt, L^2(dxMdv))
$$
and 
$$ 
\frac{1}{\eps_n^2} (\sqrt{G_{\eps_n 1}'G_{\eps_n}'}-\sqrt{G_{\eps_n 1}G_{\eps_n}})
\to \tilde q \hbox{ in }\wL^2(dtdxd\mu)\,.
$$

Then, by (\ref{Entr-estm2}) and (\ref{Entr-prd1}),
$$
\begin{aligned}
\iint Mg^2(t,x,v)dxdv
& \leq \varliminf_{n\to \infty} 4\iint M\left(\frac{\sqrt{G_{\eps_n}}-1}{\eps_n}\right)^2(t,x,v) dvdx 
\\
&\leq \varliminf_{n\to \infty} \frac{2}{\eps_n^2} H(F_{\eps_n}|M)(t),
\end{aligned}
$$
and
$$
\begin{aligned}
\int_0^t\iint\tilde q^2 dsdxd\mu 
&\leq \varliminf_{n\to \infty}\int_0^t\iint\left(
\frac{\sqrt{G_{\eps_n 1}'G_{\eps_n}'}-\sqrt{G_{\eps_n 1}G_{\eps_n}}}{\eps_n^2}\right)^2 dsdxd\mu
\\
& \leq \varliminf_{n\to \infty}\frac1{\eps_n^4}\int_0^t\int \cE(F_{\eps_n})dsdx
\end{aligned}
$$

Explicit computations based on the limiting forms (\ref{g=}) and (\ref{q-lim}) of 
$$
g\hbox{ and }\iint \tilde q b(v-v_1,\omega) d\omega M_1 dv_1
$$ 
and using the symmetries of $\tilde q$ under the $d\mu$-symmetries imply that
$$
\iint Mg^2(t,x,v)dxdv= \int (|u|^2(t,x) +\tfrac52|\theta|^2(t,x) )dx\,,
$$
while
$$
\int\tilde q^2 d\mu\geq
	\tfrac12\nu|\nabla_x u+(\nabla _x u)^T|^2 +\tfrac52\kappa|\nabla_x \theta|^2
$$
(see Lemma 4.7 in \cite{BGL2} for a detailed  proof of these statements.)

Taking limits in the scaled entropy inequality 
$$
\frac1{\eps^2} H(F_\eps|M)(t)
	+\frac1{\eps^4} \int_0^t \int \cE(F_\eps)(s,x) dxds 
		\leq \frac1{\eps^2} H(F_\eps^{in}|M)
$$
entails the expected energy inequality
$$
\begin{aligned}
\int_{\bR^3}(\tfrac12|u(t,x)|^2+\tfrac54|\th(t,x)|^2)dx
&+
\int_0^t\int_{\bR^3}(\nu|\grad_xu|^2+\tfrac52\ka|\grad_x\th|^2)dxds
\\
&\le \varliminf\frac1{\eps^2} H(F_\eps^{in}|M)
\end{aligned}
$$
With this last observation, the proof of Theorem \ref{BNSW-TH} is complete.

\setcounter{equation}{0}
\setcounter{subsection}{0}
\setcounter{Thm}{0}
\renewcommand{\theequation}{A.\arabic{equation}}
\renewcommand{\thesubsection}{A.\arabic{subsection}}

\vskip2cm

\section*{Appendix A. Some results about the limits of products}
\renewcommand{\thesection}{A}
\numberwithin{equation}{section} \numberwithin{Thm}{section}


For the sake of completeness, we recall here without proof some classical results 
used in the present paper to pass to the limit in nonlinear terms.

\bigskip
The first one is due to DiPerna and Lions \cite{DPL}, and is referred to as the 
Product Limit Theorem in \cite{BGL2}~:

\begin{Thm}\label{A1-thm}
Let $\mu$ be a finite, positive Borel measure on a Borel subset $X$ of $\bR^N$. 
Consider two sequences of real-valued measurable functions defined on $X$ 
denoted $\varphi_n$ and $\psi_n$. 

Assume that $(\psi_n)$ is bounded in $L^\infty(d\mu)$ and such that $\psi_n \to \psi$ 
a.e. on $X$ while $\varphi_n \to \varphi $ in $w-L^1(d\mu)$. Then the product
$$
\varphi_n \psi_n \to \varphi\psi\hbox{ in }L^1(d\mu)\hbox{ weak.}
$$
\end{Thm}

\bigskip
The second one is due to Lions and Masmoudi \cite{LiMa1}, and can be viewed as 
a compensated compactness result. It states that (fast oscillating) acoustic waves 
do not contribute to the macroscopic dynamics in the incompressible limit~:

\begin{Thm}\label{A2-thm}
Let $c\neq 0$. Consider two families $(\varphi_\eps)$ and $(\nabla_x \psi_\eps)$ 
bounded in $L^\infty_{loc} (dt, L^2_{loc}(dx))$, such that
$$
\begin{aligned}
\d_t \varphi_\eps+\frac1\eps \Delta_x \psi_\eps =\frac1\eps F_\eps,\\
\d_t \nabla \psi_\eps+\frac{c^2}\eps\nabla_x \varphi_\eps =\frac1\eps G_\eps,
\end{aligned}
$$
for some $F_\eps, G_\eps$ converging to 0 in $L^1_{loc}(dt, L^2_{loc}(dx))$.

Then the quadratic quantities
$$
P\nabla_x \cdot ((\nabla_x \psi_\eps)^{\otimes 2}) 
	\hbox{ and } \nabla_x \cdot (\varphi_\eps \nabla_x \psi_\eps)
$$
converge to 0 in the sense of distributions on $\bR_+^*\times \bR^3$.
\end{Thm}


\setcounter{equation}{0}
\setcounter{subsection}{0}
\setcounter{Thm}{0}
\renewcommand{\theequation}{B.\arabic{equation}}
\renewcommand{\thesubsection}{B.\arabic{subsection}}

\vskip2cm

\section*{Appendix B. Some regularity results for the  free transport operator}
\renewcommand{\thesection}{B}
\numberwithin{equation}{section} \numberwithin{Thm}{section}


The main new idea in our previous work on the Navier-Stokes limit of the Boltzmann 
equation \cite{GSRInvMath} was to improve integrability and regularity with respect 
to the $x$ variables using the integrability with respect to the $v$ variables. 

This property is obtained by combining the velocity averaging lemma \cite{GLPS, GPS} 
with dispersive properties  of the free transport operator \cite{CP}. 

We state here two results of this kind used in the present paper, whose proof can be 
found in \cite{GSR2} or \cite{GSRInvMath}.

\bigskip
The first such result, based on the dispersive properties of free transport, explains 
how the streaming operator transfers uniform integrability from the $v$ variables to the 
$x$ variables.

\begin{Thm}\label{B1-thm}
Consider  a bounded family $(\psi_\eps)$ of  $ L^\infty_{loc}(dt, L^1_{loc}(dxdv))$ 
such that $(\eps \d_t +v\cdot \nabla_x) \psi_\eps$ is bounded  in $L^1_{loc}(dtdxdv)$. 
Assume that $(\psi_\eps)$ is locally uniformly integrable in the $v$-variable --- see
Proposition \ref{UINTV-PR} for a definition of this notion. Then $(\psi_\eps)$ is locally 
uniformly integrable (in all variables $t$, $x$ and $v$).
\end{Thm}

\bigskip
The second one, which is based on the classical velocity averaging theorem in
\cite{GLPS, GPS}, explains how this additional integrability is used to prove a $L^1$ 
averaging lemma.

\begin{Thm}\label{B2-thm}
Consider  a bounded family $(\varphi_\eps)$ of  $ L^2_{loc}(dtdx, L^2(Mdv))$ 
such that $(\eps \d_t +v\cdot \nabla_x) \varphi_\eps$ is weakly relatively compact 
in $L^1_{loc}(dtdxMdv)$. Assume that $(|\varphi_\eps|^2)$ is locally uniformly 
integrable with respect to the measure $dtdxMdv$.

Then, for each function $\xi\equiv \xi(v)$ in $L^2(Mdv)$, each $T>0$ and each 
compact $K\subset \bR^3$, 
$$
\left\| \int \varphi_\eps(t,x+y, v) M\xi(v)dv 
	-\int \varphi_\eps(t,x, v) M\xi(v)dv\right\|_{L^2([0,T]\times K)}\to 0
$$
as $|y|\to 0$ uniformly in $\eps$.
\end{Thm}


\setcounter{equation}{0}
\setcounter{subsection}{0}
\setcounter{Thm}{0}
\renewcommand{\theequation}{B.\arabic{equation}}
\renewcommand{\thesubsection}{B.\arabic{subsection}}

\vskip2cm

\section*{Appendix C. Some regularity results for the Leray projection}
\renewcommand{\thesection}{C}
\numberwithin{equation}{section} \numberwithin{Thm}{section}


One annoying difficulty in handling incompressible or weakly compressible 
models is the nonlocal nature of the Leray projection $P$ --- defined on the
space $L^2(\bR^3;\bR^3)$ of square integrable vector fields, on the closed
subspace of divergence-free vector fields. By definition, $P$ is continuous
on $L^2(\bR^3;\bR^3)$, as well as on $H^s(\bR^3;\bR^3)$ --- since $P$ is
a classical pseudo-differential operator of order $0$. However, $P$ is not
continuous on local spaces of the type $L^p_{loc}(dx)$. Here is how we
make up for this lack of continuity.

\bigskip
A first observation leads to a \textit{local} $L^2$-equicontinuity statement 
provided that some \textit{global bound} is known to hold.

\begin{Lem}\label{P-equicontinuity}
Consider  a  sequence of vector fields $(\psi_n)$ uniformly bounded in 
$L^\infty_t(L^2(dx))$. Assume that, for each $T>0$ and and each compact  
$K\subset \bR^3$,
$$
\int_0^T\int_K \left| \psi_n(t,x+y) -\psi_n(t,x)\right|^2 dxdt \to 0 \hbox{ as }|y|\to 0
$$
uniformly in $n$.

Then, for each $T>0$ and and each compact  $K\subset \bR^3$,
$$
\int_0^T\int_K \left|P \psi_n(t,x+y) -P\psi_n(t,x)\right|^2 dxdt \to 0 \hbox{ as }|y|\to 0
$$
uniformly in $n$.
\end{Lem}

\begin{proof}
For each $\de\in(0,1)$ and $R>0$, let $\chi\equiv\chi(x)$ be a $C^\infty$ 
bump function satisfying
$$
\begin{aligned}
{}&\chi(x)=1\hbox{ for }|x|\le R\,,\quad&&\chi(x)=0\hbox{ for }|x|\ge R+\de\,,
\\
&0\le\chi\le 1\,,\quad&&|\chi'|\le2/\de\,.
\end{aligned}
$$
Obviously, for $|y|\le 1$, one has
$$
\begin{aligned}
\int_0^T\int_{\bR^3}|\chi(x+y) \psi_n(t,x+y)
    -\chi(x) \psi_n(t,x)|^2dxdt
\\
\le
2\int_0^T\int_{\bR^3}\chi(x+y)^2| \psi_n(t,x+y)
    - \psi_n(t,x)|^2dxdt
\\
+2\int_0^T\int_{\bR^3}|\chi(x+y)-\chi(x)|^2| \psi_n(t,x)|^2dxdt
\\
\le
2\int_0^T\int_{|x|\le R+2}| \psi_n(t,x+y)
    - \psi_n(t,x)|^2dxdt
2\left(\frac{2}{\de}\right)^2|y|^2T\|\psi_n\|_{L^\infty_t(L^2_x)}
\end{aligned}
$$
so that
$$
\int_0^T\int_{\bR^3}|\chi(x+y) \psi_n(t,x+y)
    -\chi(x) \psi_n(t,x)|^2dxdt\to 0
$$
as $|y|\to 0$ uniformly in $n$, since $\psi_n$ is bounded in $L^\infty_t(L^2(Mdvdx))$.

Consider next the decomposition
$$
\chi P=P\chi +[\chi,P]
$$
where $\chi$ denotes the pointwise multiplication by the function $\chi$, which is another
pseudo-differential operator of order 0 on $\bR^3$. In particular, $[\chi,P]$ is a classical
pseudo-differential operator of order $-1$ on $\bR^3$.

With this decomposition, we consider the expression
$$
\begin{aligned}
\int_0^T\int_{|x|\le R}|\chi(x+y)P\psi_n(t,x+y)
    -\chi(x)P\psi_n(t,x)|^2dxdt
\\
\le
2\int_0^T\int_{|x|\le R}|P(\chi\psi_n)(t,x+y)
    -P(\chi\psi_n)(t,x)|^2dxdt
\\
+2\int_0^T\int_{|x|\le R}|[\chi,P]\psi_n(t,x+y)
    -[\chi,P]\psi_n(t,x)|^2dxdt
\end{aligned}
$$
Because $P$ is an $L^2(dx)$-orthogonal projection, the first integral on the right-hand
side of the inequality above satisfies
$$
\begin{aligned}
\int_0^T\int_{|x|\le R}|P(\chi\psi_n)(t,x+y)
    -P(\chi\psi_n)(t,x)|^2dxdt
\\
\le
\int_0^T\int_{\bR^3}|P(\chi\psi_n)(t,x+y)
    -P(\chi\psi_n)(t,x)|^2dxdt
\\
\le
\int_0^T\int_{\bR^3}|\chi(x+y)\psi_n(t,x+y)
    -\chi(x)\psi_n(t,x)|^2dxdt\to 0
\end{aligned}
$$
as $|y|\to 0$, uniformly in $n$. On the other hand, because $[\chi,P]$ is a classical
pseudo-differential operator of order $-1$ on $\bR^3$ (see \cite{Stein}, \S 7.16, on
p. 268):  therefore $[\chi,P]$ maps $L^2(\bR^3)$ continuously into $H^1(\bR^3)$.
This implies in particular that $[\chi,P]\psi_n$ is bounded in $L^\infty(\bR_+;H^1(\bR^3))$ 
so that, for each $R>0$,
$$
\int_0^T\int_{|x|\le R}|[\chi,P]\psi_n(t,x+y)
    -[\chi,P]\psi_n(t,x)|^2dxdt\to 0
$$
as $|y|\to 0$, uniformly in $n$. Hence
$$
\int_0^T\int_{|x|\le R}|\chi(x+y)P\psi_n(t,x+y)
    -\chi(x)P\psi_n(t,x)|^2dxdt\to 0
$$
as $|y|\to 0$, uniformly in $n$. Assuming that $R>2$, that the parameter $\de$ in
the definition of $\chi$ satisfies $\de\in(0,1)$ and that $|y|\le 1$, we conclude  that
$$
\int_0^T\int_{|x|\le R-2}|P\psi_n(t,x+y)
    -P\psi_n(t,x)|^2dxdt\to 0
$$  
as $|y|\to 0$, uniformly in $n$, for each $R>0$ sufficiently large. 
\end{proof}

\bigskip
A second observation provides continuity of $P$ in $L^1_{loc}$ under some 
appropriate global bounds.

\begin{Lem}\label{LM-Q}
Let $\psi_\eps\equiv \psi_\eps(t,x)\in\bR^3$ be a family of vector fields such that 
$\psi_\eps\to 0$ in $L^1_{loc}(dtdx)$ and $\psi_\eps=O(1)$ in $L^1_{loc}(dt;L^2_x)$. 
Let $\xi_\de$ be a mollifying sequence on $\bR^3$ defined by
$\xi_\de(x)=\de^{-3}\xi(x/\de)$ where $\xi\in C^\infty_c(\bR^3)$ is a bump function 
such that
$$
\Supp\xi\subset B(0,1)\,,\quad\xi\ge 0\,,\hbox{ and }\int\xi dx=1\,.
$$
Then, for each $\de>0$,
$$
Q(\xi_\de\star\psi_\eps)\to 0\hbox{ in }L^1_{loc}(dtdx)\hbox{ as }\eps\to 0\,.
$$
\end{Lem}

\begin{proof}
Let $\chi\in C^\infty_c(\bR^3)$. Then
$$
\int_0^T\!\!\int_{\bR^3}\chi(x)|Q(\xi_\de\star\psi_\eps)(t,x)|dxdt
=
\int_0^T\!\!\int_{\bR^3}\chi(x)\Om(t,x)\cdot Q(\xi_\de\star\psi_\eps)(t,x)dxdt
$$
where $\Om$ is any measurable unit vector field such that
$$
\Om(t,x)=\frac{Q(\xi_\de\star\psi_\eps)}{|Q(\xi_\de\star\psi_\eps)|}(t,x)
    \hbox{ whenever }Q(\xi_\de\star\psi_\eps)(t,x)\not=0\,.
$$
Hence
$$
\begin{aligned}
\int_0^T\int_{\bR^3}\chi(x)|Q(\xi_\de\star\psi_\eps)(t,x)|dxdt
\\
=-\int_0^T\int_{\bR^3}
\Dlt_x^{-1}\Div_x\left(\chi\Om\right)\Div_x\left(\xi_\de\star\psi_\eps\right)(t,x)dxdt
\end{aligned}
$$
Let ${\mathcal G}(x)=\frac{x}{4\pi|x|^3}$ be the convolution 
kernel corresponding to $-\grad_x\Dlt_x^{-1}$; for $R>0$, 
denote ${\mathcal G}_R(x)={\mathcal G}(x)\indc_{|x|<R}$ 
and ${\mathcal G}^R(x)={\mathcal G}(x)\indc_{|x|\ge R}$. 
Thus
$$
\begin{aligned}
\int_0^T\int_{\bR^3}\chi(x)|Q(\xi_\de\star\psi_\eps)(t,x)|dxdt
\\
=
\int_0^T\int_{\bR^3}{\mathcal G}_R
	\star\left(\chi\Om\right)(\grad\xi_\de)\star\psi_\eps(t,x)dxdt
\\
+
\int_0^T\int_{\bR^3}{\mathcal G}^R\star\left(\chi\Om\right)(\grad\xi_\de)\star\psi_\eps(t,x)dxdt
\end{aligned}
$$
We have used here the following symplifying notation: if $a$ and $b$ are two 
vector fields on $\bR^3$, we denote
$$
a\star b(x)=\int_{\bR^3}a(x-y)\cdot b(y)dy
$$
where $\cdot$ is the canonical inner product on $\bR^3$.

Observe that ${\mathcal G}^R=O(1/\sqrt{R})$ in $L^2_x$, while $\chi\Om\in L^\infty_t(L^1_x)$ 
(since $|\Om|=1$ and $\Supp(\chi)$ is compact). Hence 
$$
{\mathcal G}^R\star\left(\chi\Om\right)=O(1/\sqrt{R})\hbox{ in }L^1_{loc}(dt;L^2_x)
$$ 
and $(\grad\xi_\de)\star\psi_\eps=O(1)$ in $L^1_{loc}(dt;L^2_x)$ for each $\de>0$ since 
$\psi_\eps=O(1)$ in $L^1_{loc}(dt;L^2_x)$. Hence the second integral is $O(1/\sqrt{R})$ 
for each $\de>0$.

Next ${\mathcal G}_R=O(R)$ in $L^1_x$ and thus 
${\mathcal G}_R\star\left(\chi\Om\right)=O(R)$ in $L^\infty_x$ since $|\Om|=1$; moreover, 
$$
\Supp_x\left({\mathcal G}_R\star\left(\chi\Om\right)\right)\subset\Supp(\chi)+B(0,R)
$$ 
is bounded for each $R>0$. On the other hand $\grad\xi_\de\star\psi_\eps\to 0$ in 
$L^1_{loc}(dtdx)$, so that the first integral vanishes as $\eps\to 0$ for each $\de>0$ 
and each $R>0$. Passing to the limsup as $\eps\to 0^+$, then letting $R\to 0^+$
leads to the announced result.
\end{proof}


\end{document}